\documentclass[11pt]{amsart}

\usepackage{mystyle}

\begin{document}

\title{H\"older estimates for non-local parabolic equations with critical drift}

\author[H. A. Chang-Lara]{H\'ector A. Chang-Lara}
\address{Department of Mathematics, Columbia University, New York, NY 10027}
\email{changlara@math.columbia.edu}

\author[G. D\'avila]{Gonzalo D\'avila}
\address{Department of Mathematics, University of British Columbia, Vancouver, BC V6T 1Z2}
\email{gdavila@math.ubc.ca}

%\begin{abstract}
%\end{abstract}
%\subjclass{}
%\keywords{}
\begin{abstract}
In this paper we extend previous results on the regularity of solutions of integro-differential parabolic equations. The kernels are non necessarily symmetric which could be interpreted as a non-local drift with the same order as the diffusion. We provide an Oscillation Lemma and a Harnack Inequality which can be used to prove higher regularity estimates.
\end{abstract}
\maketitle

%%%Introduction%%%
\section{Introduction}\label{section:Introduction}
We are interested in studying regularity properties for time evolution problems, driven by fully nonlinear integro-differential operators $I$ of order $\s \in [1,2)$, with local and non-local drift to be defined, 
\begin{align*}
u_t - Iu = f(t). 
\end{align*}
To keep an example in mind, consider that the operator $I$ might be given by a combination of linear operators with non-symmetric kernels,
\begin{align*}
Iu(x) &= \inf_\b\sup_\a(2-\s)\int \d u(x;y)\frac{K_{\a,\b}(y)}{|y|^{n+\s}}dy + b_{\a,\b}\cdot Du(x),\\
\d u(x;y) &:= u(x+y) - u(x) - Du(x)\cdot y\chi_{B_1}(y).
\end{align*}

Non-local equations are a subject that has been study extensively in the last decade, from both the probabilistic and the analytic approach. From the probabilistic side, regularity of solutions for the stationary problem has been studied in \cite{Kassmann05}, \cite{Kassmann05-2} and \cite{Bass02}, where they prove Harnack inequalities and therefore H\"older estimates. However these results degenerate when the order of the equation goes to the classical one.

The first uniform estimates on the order equation are due to L. Caffarelli and L. Silvestre in the elliptic case. In a series of papers \cite{Caffarelli09}, \cite{Caffarelli11} and \cite{Caffarelli12}, they studied the regularity of the solutions of fully nonlinear non-local elliptic equations extensively, proving under different hypothesis $C^\a$, $C^{1,\a}$ and $C^{\s+\a}$ estimates. The approach was purely non-variational and used tools like the Alexandroff-Bakelman-Pucci (ABP) and the Point Estimate.

In the parabolic setting, the variational problem was studied by L. Caffarelli, C. Chan and A. Vasseur in \cite{Vasseur11} by using De Giorgi's technique. Also, M. Felsinger and M. Kassmann in \cite{Kassmann12}, obtained a Harnack inequality where the constants remain uniform as the order of the equation goes to the classical one by using Moser's technique. We point out that both of these papers derive the equation from a variational principle and would be the equivalent to the regularity theory of parabolic equations in divergence form.

In the fully nonlinear setting, L. Silvestre studied in \cite{Silvestre11} the regularity of solutions to a Hamilton-Jacobi equation with critical fractional diffusion where the order of the equation is one. His work uses a non-variational approach to proof a Diminish of Oscillation estimate.

In the case when there is no drift, $b_{\a,\b}=0$ and $K$ is even for the equation considered above, the authors prove in \cite{Davila12-p} that solutions of the equation are H\"older continuous in space and time by combining the techniques from \cite{Caffarelli09}, \cite{Silvestre11} and \cite{Wang92}. In the translation invariant case, this implies $C^{1,\a}$ estimates in space under smoothness hypothesis for the kernels. In \cite{Chang-Davila2} the authors extend most of these results for the non-translation invariant case for equations with regular coefficients. Recently, Jin, T. and Xiong, J. considered in \cite{Jin14} higher order, optimal Schauder estimates  for linear operators. Also recently, improvements by J. Serra in \cite{Serra14} allowed to remove the smoothness requirement for the symmetric kernels to obtain the $C^{1,\a}$ estimate in space. This work, include the analogous result for the non-symmetric kernels following the techniques from \cite{Serra14}. 

In this paper we are concerned with studying the case without symmetry assumptions on $K$. As we will discuss in the following section, the scaling suggest to include gradient terms. In the second order theory one can argue that at small scales the drift, which has order one, becomes sufficiently small such that it can be absorbed by estimates that can be proved for pure second order equations, this is in fact the approach taken by the authors of this paper in \cite{Davila12}. In present work the drift may still be comparable to the diffusion as the scales approach zero giving us a critical type of problem, similar to the one treated in \cite{Chang12} for the elliptic case. We include new estimates as the Oscillation Lemma and the Harnack inequality. They are important in order to get an analogous of the non-local Evans-Krylov theorem in the parabolic setting in a coming paper, see also \cite{Caffarelli12}, \cite{Serra14-2} and \cite{Jin14-2} for the elliptic case.

The paper is arranged as follows. In Section \ref{section:Preliminaries} we introduce some standard notation, define the operators, the notion of solution and state some basic results. In Section \ref{section:Qualitative Properties} we study some qualitative properties including the stability, a comparison principle and the existence and uniqueness of solutions.  The familiar reader may want to skip them. A non-local version of the Alexandroff-Bakelman-Pucci estimate is proven in Section \ref{section:Alexandroff-Bakelman-Pucci type of estimate}. On Section \ref{section:Point Estimate} we prove a Point Estimate. We derive in Section \ref{section:Oscillation Lemma and Harnack Inequality} an Oscillation Lemma and Harnack Inequality uniform in the order of the equation. As consequences we obtain H\"older estimates for the solutions in Section \ref{section:regularity}.

%%% Preliminaries %%%

\section{Preliminaries}\label{section:Preliminaries}

The cylinder of radius $r$, height $\t$ and center $(x,t)$ in $\R^n\times\R$ is denoted by $C_{r,\t}(x,t) := B_r(x)\times(t-\t,t]$. The cube of side length $r$ and center $x$ in $\R^n$ by $Q_r(x) := (x_1-r/2,x_1+r/2)\times\ldots\times(x_n-r/2,x_n+r/2)$. The box of side $r$, height $\t$ and centered at $(x,t)$ in $\R^n\times\R$ is denoted by $K_{r,\t}(x,t) := Q_r(x)\times(t-\t,t]$. Whenever we omit the center we are assuming that they get centered at the origin in space and time.

The parabolic topology on $\R^n\times\R$ consists of the one generated by neighborhoods of the form $C_{r,\t}(x,t)$ with respect to the point $(x,t)$. We use $(x_i,t_i) \to (x,t^-)$ to denote a sequence converging to the point $(x,t)$ with respect to this topology. In particular,
\begin{align*}
u_{t^-}(x,t) := \lim_{\t\searrow 0} \frac{u(x,t) - u(x,t-\t)}{\t}.
\end{align*}

The parabolic non-local boundary, suitable for our Dirichlet problem on a domain $\W\times(t_1,t_2]$, $\Omega\subset\R^n$, is
\begin{align*}
\p_p(\W\times(t_1,t_2]) := (\W^c\times(t_1,t_2]) \cup (\R^n \times \{t_1\}).
\end{align*}

The weighted space $L^1(\w_\s)$ with respect to
\begin{align*}
\w_\s(y) := \min(1,|y|^{-(n+\s)}),
\end{align*}
consists of all measurable functions $u:\R^n \to \R$ such that
\begin{align*}
\|u\|_{L^1(\w_\s)} := \int_\R |u(y)|\w_\s(y)dy < \8.
\end{align*}

%%% non-local Elliptic Operators %%%

\subsection{Non-local Uniformly Elliptic Operators}

Given $\s \in (0,2)$, a measurable kernel $K:\R^n\to[0,\8)$ and a vector $b \in \R^n$, the non-local linear operator $L_{K,b}^\s$ is defined by
\begin{align}\label{eq:linear_operators}
L_{K,b}^\s u(x) &:= (2-\s)\int \d u(x;y)\frac{K(y)dy}{|y|^{n+\s}} + b\cdot Du(x),\\
\nonumber \d u(x;y) &:= u(x+y) - u(x) - Du(x)\cdot y\chi_{B_1}(y).
\end{align}
We may also use $L_K^\s$ and $b\cdot D$ for $L_{K,0}^\s$ and $L_{0,b}^\s$ respectively.

Given that $u$ is sufficiently smooth and integrable ($L^1(\w_\s)$ for the tail), it suffices that $K$ is bounded from above for the convolution integral to converge. On the other hand, we will see that the operator has enough diffusion if it is bounded away from zero.

\begin{definition}
Let $\cK_0 = \cK_0(\l,\L)$ be the family of measurable kernels satisfying,
\begin{align*}
0<\l \leq K \leq \L<\8.
\end{align*}
\end{definition}

Non-linear operators are now obtained as an arbitrary combination of linear operators which may vary from point to point in the domain $\W\times(t_1,t_2]$.

\begin{definition}
Given $\cL \ss \{L_{K,b}^\s\}_{K \in \cK_0, b \in \R^n}$, a function $I:\W\times(t_1,t_2]\times\R^{\cL} \to \R$ determines a non-local operator of order $\s$ by,
\begin{align*}
Iu(x,t) := I\1x,t,(Lu(x,t))_{L\in\cL}\2.
\end{align*}
\end{definition}

%When there is no chance of confusion we write $(L_{K,b}^\s)$ instead of $(L_{K,b}^\s)_{(K,b)\in\cL}$ and $I(x,t,L_{K,b}^\s) \in \W\times(t_1,t_2]\times\R^{\cL}$ instead of $I(x,t,(L_{K,b}^\s)_{(K,b)\in\cL})$ and $(L_{K,b}^\s)_{(K,b)\in\cL}$. We also denote $l_{K,0}$ and $l_{0,b}$ by $l_{K}$ and $p_b$ respectively.

We denote vectors in $\R^\cL$ by $(l_L)_{L \in \cL}$ or just by the abbreviation $(l_L)$ whenever it is clear from the context. The linear operator $L$ takes the role of an index and for each one of them, $l_L$ is a real number. Keep in mind the analog with pure second order equations obtained from the hessian which encodes all possible second order derivatives. In our case, given $u$, there is no finite set of numbers encoding the same information for non-local operators applied to $u$. In some sense, $(Lu)_{L\in\cL}$ is a type of hessian which whenever gets evaluated at a particular point gives a set of numbers $(l_L)_{L\in\cL} \in \R^\cL$, indexed by $L$. Same as for the hessian matrix which is symmetric, there might be some redundancy in the vector $(l_L)$ whenever it corresponds to the evaluations $l_L = Lu$.

We say $I$ is (degenerate) elliptic if it is (non-decreasing) increasing in $(l_L)\in \R^\cL$. $I$ is translation invariant in space or time if the function $I$ does not depend on the variable $x$ or $t$. Translation invariant, without making reference the space or time variable, means that it is translation invariant with respect to both. Finally, $I$ is (semi)continuous if the function $I$ is (semi)continuous when $\W\times(t_1,t_2]\times\R^{\cL}$ is equipped with the $L^\8$ norm.

% Scaling %

\subsubsection{Scaling}

An important ingredient in regularity theory is scale inva-riance. A diminish of oscillation estimate implies the regularity of the solution because they can be iterated at smaller scales.

Let $\s\in(0,2)$, and $u$ satisfying the non-homogeneous linear equation without gradient term,
\begin{align*}
u_t - L_K^\s u = f \text{ in $\W\times(t_1,t_2]$},
\end{align*}
we consider a rescaling of the form $\tilde u(x,t) := r^{-\s}u(rx,r^\s t)$ with $r\in(0,1)$. By the change of variable formula it satisfies in $r^{-1}\W\times(r^{-\s} a,r^{-\s} b]$,
\begin{align}\label{scaling}
\tilde u_t - L^\s_{K(r\cdot)} \tilde u - \1r^{\s-1}(2-\s)\int_{B_1\sm B_r} \frac{yK(y)}{|y|^{n+\s}}dy\2\cdot D\tilde u = f(r\cdot,r^\s\cdot).
\end{align}
It comes immediately to our attention the gradient term, which was not explicitly present in $L^\s_K$ and depends on the odd part of $K$. This explains why we included the gradient variable in $I$. Also the diffusion, contained now in $L^\s_{K(r\cdot)}$, competes against the drift term if $\s \in [1,2)$.

There is a distinction with the classical second order equations with gradient terms. Usually in these cases one can argue that the drift diminishes at smaller scales and therefore can be absorbed in the regularity estimates that can be proved for equations with pure diffusion. However, in the case of the present work the drift coming from the odd part of the kernel may persist as the scaling goes to zero, competing in a critical way with the diffusion.

Another example of critical problems was also considered by L. Silvestre in \cite{Silvestre11}. Let $\cK_e \ss \cK_0$ being defined as the family of all the kernels $K \in \cK_0$ which are even, $K(y) = K(-y)$. For each $L \in \cL \ss \{L_{K,b}^\s\}_{K \in \cK_e, |b|\leq \b}$, the scaling of the drift term and the non-local diffusion get decoupled but, as before, the diffusion competes against the drift at smaller scales only if $\s \geq 1$. In particular, the case $\s=1$ is considered to be critical and includes the following Hamilton-Jacobi type of equation related with the quasi-geostrophic model,
\begin{align}\label{eq:hamilton_jacobi}
u_t - \D^{1/2}u - |Du| = 0.
\end{align}
These cases are also contained in the present work.

Lets go back to the scaling of $L_{K,b}^\s$. In order to have a bounded drift at small scales we need to assume that
\begin{align*}
\sup_{r\in(0,1)}r^{\s-1}\left|b + (2-\s)\int_{B_1\sm B_r} \frac{yK(y)}{|y|^{n+\s}}dy\right| < \8.
\end{align*}

\begin{definition}\label{def:linear_family}
Given $\s \in [1,2)$, let $\cL_0^\s(\l,\L,\b)$ be the family of linear operators $L_{K,b}^\s$ such that $K \in \cK_0(\l,\L)$ and,
\begin{align*}
\sup_{r\in(0,1)}r^{\s-1}\left|b + (2-\s)\int_{B_1\sm B_r} \frac{yK(y)}{|y|^{n+\s}}dy\right| \leq \b.
\end{align*}
\end{definition}

When the parameters have been fixed we denote $\cL_0$ for $\cL_0^\s(\l,\L,\b)$ or may simply highlight the parameters which are relevant for the discussion.

For $\s>1$ the control over the integral of the odd part of the kernel follows from the fact $K$ is bounded. In this case we just need to bound the drift $b$ and consider $\b$ sufficiently large with respect to $\L$. However, for $\s=1$, if we just assume $K$ bounded, then the previous integral might diverge with a logarithmic rate.

We say that $\cL$ is scale invariant if whenever $L_{K,b}^\s \in \cL$ then also $L_{K^r,b^r}^\s \in \cL$ for $r>0$ and,
\begin{align*}
K^r(y) &:= K(ry),\\
b^r &:= \begin{cases}
\displaystyle r^{\s-1}\1b + (2-\s)\int_{B_1 \sm B_r} \frac{y K(y)dy}{|y|^{n+\s}}\2 &\text{ if $r \leq 1$},\\
\displaystyle r^{\s-1}\1b - (2-\s)\int_{B_r \sm B_1} \frac{y K(y)dy}{|y|^{n+\s}}\2 &\text{ if $r > 1$}.
\end{cases}
\end{align*}
For instance, $\cL_0$ is scale invariant.

We describe now the type of equations that are obtained from an equation for $u$ by standard transformations used frequently in the theory. Let $\cL$ be scale invariant, $I:\W\times(t_1,t_2]\times\R^\cL \to \R$ and $u$ satisfies an equation,
\begin{align*}
u_t - Iu = f \text{ in $\W\times(t_1,t_2]$}.
\end{align*}
Given,
\begin{align*}
\tilde u(x,t) &:= \1\frac{u-\varphi}{C}\2(rx+x_0,r^\s t+t_0), \text{ ($\varphi$ smooth/integrable),}\\
\tilde I(x,t,l_{L_{K,b}}) &:= \frac{r^\s}{C}I\1rx+x_0,r^\s t+t_0,\frac{C}{r^\s}l_{L_{K^r,b^r}}+(L_{K,b}\varphi)(rx+x_0,r^\s t+t_0)\2,\\
\tilde f(x,t) &:= f(rx,r^{\s}t).
\end{align*}
Then, $\tilde u$ satisfies,
\begin{align*}
\tilde u_t - I^r\tilde u = \tilde f \text{ in $\frac{\W-x_0}{r}\times\left(\frac{t_1-t_0}{r^\s}, \frac{t_2-t_0}{r^\s}\right]$}.
\end{align*} 

\begin{definition}[Uniformly Ellipticity]
For $\cL \ss \cL_0^\s(\l,\L,\b)$ scale invariant and $I:\W\times(t_1,t_2]\times\R^\cL\to\R$, we say that $I$ is uniformly elliptic if for every $(x,t) \in \W\times(t_1,t_2]$ and $\1l^{(1)}_L\2, \1l^{(2)}_L\2 \in \R^\cL$,
\begin{align}\label{eq:uniform_ellipticity_I}
\inf_{L\in\cL}\1 l^{(1)}_L-l^{(2)}_L\2 &\leq I\1x,t,l^{(1)}_L\2 - I\1x,t,l^{(2)}_L\2 \leq \sup_{L\in\cL}\1l^{(1)}_L-l^{(2)}_L\2.
\end{align}
\end{definition}

%Whenever we omit the family $\cL$ we are using that $I$ is uniformly elliptic at least with respect to $\cL_0$.

The uniform ellipticity identities imply that $I$ is Lipschitz in $\R^\cL$, uniformly in $\W\times(t_1,t_2]$ and with respect to the $L^\8$ norm,
\begin{align*}
\sup_{(x,t)\in\W\times(t_1,t_2]}\left|I\1x,t,l^{(1)}_L\2 - I\1x,t,l^{(2)}_L\2\right| &\leq \left\|\1l^{(1)}_L-l^{(2)}_L\2\right\|_\8.
\end{align*}

Given that $I$ is uniformly elliptic and $\tilde I$ is constructed as in the expression above the previous definition, then also $\tilde I$ is uniformly elliptic with respect to the same constants.

% Examples %

\subsubsection{Examples}

For $\s\in(0,2)$, a fractional power of the laplacian $\D^{\s/2} = -(-\D)^{\s/2}$ is defined as the linear operator with constant kernel $K_{\D^{\s/2}}(y) := C_{n,\s}$. The constant $C_{n,\s}$ is used to have the following identity on the Fourier side, $\widehat{(-\D)^\s} = |\xi|^\s$. For $\s \in [1,2)$, $C_{n,\s}/(2-\s)$ remains uniformly bounded from above and away from zero.

Linear operators with variable coefficients are those defined as in \eqref{eq:linear_operators} replacing $K(y)$ and $b$ by $K(x,t;y)$ and $b(x,t)$ respectively. They can be clearly obtained from $I:\W\times(t_1,t_2]\times\R^\cL\to\R$.

Whenever $I$ splits as $I(x,t,l_{L_{K,b}^\s}) =  V(x,t,l_{L_{K,0}^\s}) + H(x,t,L_{0,b}^\s)$, $H$ can be considered as a Hamiltonian depending on the gradient and $V$ is the viscosity term. For example, $I(l_{L_{K,b}^\s}) = l_{\D^{1/2}} + \sup_{|b| \leq 1} |l_{L_{0,b}}|$ gives the operator in the critical equation \eqref{eq:hamilton_jacobi}.

Operators obtained by inf and sup combinations of linear operators are relevant for stochastic optimal control models and also in our discussions. Lets introduce some notation.

\begin{definition}[Extremal Operators]
The extremal operators $\cM^\pm_\cL$ with respect to a family $\cL \ss \cL_0$ are defined by $\cM^-_\cL u := \inf_{L\in\cL} Lu$ and $\cM^+_\cL u := \sup_{L\in\cL} Lu$.
\end{definition}

Whenever there is no (classical) drift term, $\cL = \{L_{K,0}\}_{K \in \cK}$, we denote $\cM^\pm_\cL = \cM^\pm_\cK$. For example, the operators with respect to $\cK_0$ can be explicitly written as,
\begin{align*}
\cM^-_{\cK_0} u = (2-\sigma)\int \frac{\l\d^+u - \L\d^-u}{|y|^{n+\s}}dy,\quad \cM^+_{\cK_0} u = (2-\sigma)\int \frac{\L\d^+u - \l\d^-u}{|y|^{n+\s}}dy.
\end{align*}
where $\d u = \d^+u - \d^-u$ is the sign decomposition of $\d u$. Also,
\begin{align*}
\cM^-_{\cL_0} u \geq \cM^-_{\cK_0} u - \b|Du|,\quad \cM^+_{\cL_0} u \leq \cM^+_{\cK_0} u + \b|Du|.
\end{align*}
The equality might not hold because of the hypothesis for the non-local drift in the Definition \ref{def:linear_family}. Contrasting to the case with even kernels, $\cM_{\cK_0}^\pm$ are not scale invariant operators, however $\cM_{\cL_0}^\pm$ are.

The uniform ellipticity of $I$ with respect to $\cL$ will be frequently used in terms of sufficiently smooth/integrable functions $u,v$ in the following way,
\begin{align*}
\cM_{\cL}^-(u-v) \leq Iu - Iv \leq \cM_{\cL}^+(u-v).
\end{align*}

% Limit as $\s\to2$ $

\subsubsection{Limit as $\s \nearrow 2$}

Given that as $\s \nearrow 2$,
\begin{align}\label{eq:lim_k}
(2-\s)\int_{B_1}\frac{y\otimes yK(y)dy}{|y|^{n+\s}} \to A_K
\end{align}
we get that $L_K^\s u \to \frac{1}{2}\trace(AD^2u)$ with a modulus of convergence depending on, the modulus of convergence of the second order difference $\d u(y) \to (1/2)\trace(D^2u y\otimes y)$, the modulus of convergence of \eqref{eq:lim_k} and $\|K\|_\8$.

The limit \eqref{eq:lim_k} holds if $K(r\cdot) \to K_0$ in $L^1(\p B_1)$ as $r\searrow0$. Then $A_K$ can be explicitly computed by,
\begin{align*}
A_K = \int_{\p B_1} \theta\otimes\theta K_0(\theta)d\theta.
\end{align*}

Let $\cK = \{K\}$ be a set of kernels such that the limit \eqref{eq:lim_k} converges in a uniform way,
\begin{align*}
\lim_{\s\nearrow2} \sup_{K \in \cK} \left|(2-\s) \int_{B_1}  \frac{y\otimes yK(y)dy}{|y|^{n+\s}} - A_K\right| = 0.
\end{align*}
A function $I\in C(\R^{\cK})$, defines an operator $I_\s$ of order $\s \in (0,2)$, by
\begin{align*}
I_\s u := I((L_K^\s u)_{K \in \cK}).
\end{align*}
As $\s\nearrow2$ we obtain that
\begin{align*}
I_\s u \to I_2 u := I\1\1\frac{1}{2}\trace(A_KD^2u)\2_{K \in \cK}\2.
\end{align*}

A useful pair of examples when building barriers are the limits of $\cM^\pm_{\cK_0}$.

\begin{proposition}\label{pro:Pucci_sigma_to_two}
Given $u \in C^2 \cap L^1(\w_{\s_0})$ for some $\s_0 \in (0,2)$,
\begin{align*}
\lim_{\s\nearrow2} \cM^-_{\cK_0}u = \int_{\p B_1} \1\1\theta^t D^2 u\theta\2^+ \l - \1\theta^t D^2 u\theta\2^- \L \2d\theta,\\
\lim_{\s\nearrow2} \cM^+_{\cK_0}u = \int_{\p B_1} \1\1\theta^t D^2 u\theta\2^+ \L - \1\theta^t D^2 u\theta\2^- \l \2d\theta.
\end{align*}
In particular, these second order operators are comparable to the classical extremal Pucci operators. For some universal $C \geq 1$ depending only on the dimension,
\begin{align*}
\inf_{A\in[\l/C,C\L]} \trace (AM) &\leq \int_{\p B_1} \1\1\theta^t M\theta\2^+ \l - \1\theta^t M\theta\2^- \L \2d\theta\\
 &\leq \inf_{A\in[\l,\L]} \trace (AM),\\
\sup_{A\in[\l,\L]} \trace (AM) &\leq \int_{\p B_1} \1\1\theta^t M\theta\2^+ \L - \1\theta^t M\theta\2^- \l \2d\theta\\
&\leq \sup_{A\in[\l/C,C\L]} \trace (AM).
\end{align*}
\end{proposition}

%%% Viscosity Solutions %%% 

\subsection{Viscosity Solutions}

The set of test functions we are about to define impose sufficient requirements in order to evaluate the previous non-local operator on a cylinder $C_{r,\t}(x,t)$. First there is a condition on the continuity of the tails in time.

\begin{definition}
The space $LSC((t_1,t_2] \mapsto L^1(\w_\s))$ consists of all measurable functions $u:\R^n\times(t_1,t_2] \to \R$ such that for every $t \in (t_1,t_2]$,
\begin{enumerate}
\item $\|u(\cdot,t)^-\|_{L^1(\w_\s)} < \8.$
\item $\lim_{\t\nearrow0}\|(u(\cdot,t) - u(\cdot,t-\t))^+\|_{L^1(\w_\s)} = 0$.
\end{enumerate}

Similarly, $u \in USC((t_1,t_2] \mapsto L^1(\w_\s))$ if $-u \in LSC((t_1,t_2] \mapsto L^1(\w_\s))$ and $C((t_1,t_2] \mapsto L^1(\w_\s)) = LSC((t_1,t_2] \mapsto L^1(\w_\s)) \cap USC((t_1,t_2] \mapsto L^1(\w_\s))$.
\end{definition}

\begin{definition}[Test functions]\label{def:test_function}
A lower semicontinuous test function is defined as a pair $(\varphi, C_{r,\t}(x,t))$, such that $\varphi \in C^{1,1}_xC^1_t(C_{r,\t}(x,t)) \cap LSC((t-\t,t]\mapsto L^1(\w_\s))$. Similarly,  $(\varphi, C_{r,\t}(x,t))$ is an upper semicontinuous test function if the pair $(-\varphi, C_{r,\t}(x,t))$ is a 
lower semicontinuous test function.
\end{definition}

Test functions not only have enough regularity to evaluate $I\varphi$ but also to make it semicontinuous.

\begin{property}
Given $\cL \ss \{L^\s_{K,b}\}_{K\in\cK_0,|b|\leq \b}$, $I \in LSC(C_{r,\t}(x,t)\times\R^\cL)$ and a lower semicontinuous test function $(\varphi,C_{r,\t}(x,t))$, then $I\varphi \in LSC(C_{r,\t}(x,t))$.
\end{property}

The idea to show the semicontinuity in space or time is the same. One needs to show that $\{L_{K,b}\varphi\}_{(K,b)\in\cK_0\times \{|b|\leq\b\}}$ has a uniform modulus of semicontinuity in space and time.

Whenever the cylinder in the Definition \ref{def:test_function} becomes irrelevant we will refer to the test function $(\varphi, C_{r,\t}(x,t))$ just by $\varphi$.

\begin{definition}
Given a function $u$ and a test function $\varphi$, we say that $\varphi$ touches $u$ from below at $(x,t)$ if,
\begin{enumerate}
\item $\varphi(x,t)=u(x,t)$,
\item $\varphi(y,s) \leq u(y,s)$ for $(y,s)\in \R^n\times(t-\t,t]$.
\end{enumerate}
Similarly, $\varphi$ touches $u$ from above at $(x,t)$ if $-\varphi$ touches $-u$ from below at $(x,t)$. Finally, $\varphi$ strictly touches $u$ from above or below at $(x,t)$ if the inequality becomes strict outside of $(x,t)$.
\end{definition}

\begin{definition}[Viscosity (super) solution]\label{viscosity}
Given an elliptic operator $I$ and a function $f$, a function $u \in LSC(\W\times(t_1,t_2]) \cap LSC((t_1,t_2]\mapsto L^1(\w_\s))$ is said to be a viscosity super solution to $u_t - Iu \geq f$ in $\W\times(t_1,t_2]$, if for every lower semicontinuous test function $(\varphi,C_{r,\t}(x,t))$ touching $u$ from below at $(x,t) \in \W\times(t_1,t_2]$, we have that $\varphi_{t^-}(x,t) - I\varphi(x,t) \geq f(x,t)$.
\end{definition}

%We also write that ``$u_t - Iu \geq f$ in viscosity in $\W\times(t_1,t_2]$'' whenever $u$ is a viscosity super solution to $u_t - Iu \geq f$ in $\W\times(t_1,t_2]$.

The definition of $u$ being a viscosity sub solution to $u_t - Iu \leq f$ in $\W\times(t_1,t_2]$ is done similarly to the definition of super solution replacing $LSC$ by $USC$, contact from below by contact from above and reversing the last inequality. %Again, we say that ``$u_t - Iu \leq f$ in viscosity in $\W\times(t_1,t_2]$'' whenever $u$ is a viscosity sub solution to $u_t - Iu \leq f$ in $\W\times(t_1,t_2]$.

Finally, a viscosity solution to $u_t - Iu = f$ in $\W\times(t_1,t_2]$ is a function which is a super and a sub solution simultaneously.% and we also denote it by saying that ``$u_t - Iu = f$ in viscosity in $\W\times(t_1,t_2]$''.

%From now on, we will write that $u$ is a sub, super or solution to $u_t - Iu = f$ in $\W\times(t_1,t_2]$, and it will be always in the viscosity sense.

The requirement for the functions to be semicontinuous in time in an integral sense can be illustrated by the following example. Consider the fractional heat equation $u_t - \D^{\s/2} u = 0$ in $\W\times(t_1,t_2]$ with initial and boundary data equal to zero. It is solved classically by $u$ being identically zero in $\R^n\times[t_1,t_2]$. By modifying the boundary data at $t=t_2$, we obtain that $u$ still solves the equation in $\W\times(t_1,t_2)$ but not necessarily at $t=t_2$. Therefore, we can not expect classical solutions in this situation.

%\begin{remark}\label{rmk:test_functions}
%Given a function $u \in LSC((t-\t,t]\mapsto L^1(\w_\s)) \cup USC((t-\t,t]\mapsto L^1(\w_\s))$ and $\varphi \in C^{1,1}_xC^1_t(C_{r,\t}(x,t))$ we define the test function $(\varphi_u,C_{r,\t}(x,t))$ as
%\begin{align*}
%\varphi_u = \begin{cases}
%\varphi \text{ in $C_{r,\t}(x,t)$},\\
%u \text{ in $\p_p C_{r,\t}(x,t)$}.
%\end{cases}
%\end{align*}
%This is an admissible test function to test against $u$ when we only have a smooth function touching $u$ in a small neighborhood of the contact point.
%\end{remark}
%
%\begin{remark}\label{rmk:strict_contact}
%Sometimes we need to allow some room in the estimations by assuming that $\varphi$ strictly touches $u$ at $(x,t)$. This could be treated by adding or subtracting the following perturbation $\psi_{\d,x,t}$ to $\varphi$
%\begin{align*}
%\psi_{\d,x,t}(y,s) = \d
%\begin{cases}
%(|y-x|^2-(s-t)) &\text{ for $(y,s) \in C_{r,\t}(x,t)$},\\
%1 &\text{ for $(y,s) \in \p_p C_{r,\t}(x,t)$}.
%\end{cases}
%\end{align*}
%Given that $I$ is (semi)continuous, this small perturbation adds an error that can be sent to zero at the end of the proof. For this reason, we will assume sometimes wiouth lost of generality that a test function strictly touches the viscosity sub or super solution.
%\end{remark}

\begin{property}\label{pro:viscosity_solution}
Let $I,J$ be elliptic operators and suppose $u$ satisfies in the viscosity sense,
\begin{alignat*}{2}
u_t - Iu &\geq f &&\text{ in $\W\times(t_1,t_2]$},
\end{alignat*}
Then:
\begin{enumerate}
\item Given that $v$ satisfies in the viscosity sense
\begin{alignat*}{2}
v_t - Iv &\geq g &&\text{ in $\W\times(t_1,t_2]$},
\end{alignat*}
then for $w = \min(u,v)$ and $h = f\chi_{u<v} + g\chi_{v<u} + \max(f,g)\chi_{u=v}$ we also have that, also in the viscosity sense,
\begin{align*}
w_t - Iw \geq h \text{ in $\W\times(t_1,t_2]$}.
\end{align*}
\item Given that $I \leq J$ and $g \leq f$, then $u$ also satisfies in the viscosity sense,
\begin{alignat*}{2}
u_t - Ju &\geq g &&\text{ in $\W\times(t_1,t_2]$}.
\end{alignat*}
\item Given that $I$ is uniformly elliptic and $(\varphi, C_{r,\t}(x,t))$ is a lower semicontinuous test function touching $u$ from below at some point $(x,t)$, then the following quantities are well defined for $p = D\varphi(x)$,
\begin{align*}
L_{K,b}^p u(x) &:= \lim_{\e\to0} (2-\s)\int_{B_\e^c} \d^p u(x;y)\frac{K(y)}{|y|^{n+\s}}dy + b\cdot p,\\
\d^p u(x;y) &:= u(x+y) - u(x) - p\cdot y\chi_{B_1}(y),
\end{align*}
and they satisfy
\begin{align*}
\varphi_{t^-}(x,t) - I\1x,t,L_{K,b}^p u(x,t)\2 \geq f(x,t).
\end{align*}
\end{enumerate}
\end{property}

The first two properties are immediate from the definition. The idea of the proof for the last one is to test $u$ with a family of test functions $\varphi_{u,\e}$ that incorporates the values of $u$ outside of a small cylinder $C_{\e,\e}(x,t)$ and closes the principal value of the integral as $\e\to0$. There will be two ways to control the convergence of the integrals, one coming from the equation and the other one by the contact from below by a regular function. The uniform ellipticity is used to control the errors by the Lipschitz modulus of continuity of $I(x,t,\cdot)$. See \cite{Chang12} for the complete details.

A consequence of the previous property is the fact that sufficiently regular viscosity solutions are classical solutions.

\begin{property}
Let $I$ be a uniformly elliptic operator, $f$ be a continuous function and $u \in C^{1,\s-1+\e}_xC^1_t(\W\times(t_1,t_2])\cap USC((t_1,t_2]\mapsto L^1(\w_\s))$ such that,
\begin{align*}
u_t - Iu \leq f \text{ in viscosity in $\W\times(t_1,t_2]$}.
\end{align*}
Then it also holds that classically,
\begin{align*}
u_t - Iu \leq f \text{ in $\W\times(t_1,t_2]$}.
\end{align*}
\end{property}

The idea is that the set of points where $u$ can be touched from above is dense, therefore by the previous proposition, the equation holds in a dense set in both senses. The regularity of $u$ then implies that the equation holds at every point.

%%% Qualitative Properties %%%

\section{Qualitative Properties}\label{section:Qualitative Properties}

We treat the stability of the equations by $\G$-convergence, the Maximum Principle, the uniform ellipticity identity for viscosity solutions and the Comparison Principle. All of them lead us to the existence and uniqueness of viscosity solutions by Perron's method, provided we have barriers to force the solution to attain the boundary and initial values in a continuous way.

%%% Stability %%%

\subsection{Stability}

Viscosity sub and super solutions are stable in an appropriated notion of uniform convergence from one side. The convergence of the operators is defined by duality with respect to test functions.

\begin{definition}[Weak convergence of operators]
A sequence of operators $I_i:\W\times(t_1,t_2]\times\R^{\cL}\to\R$ converges weakly to an operator $I:\W\times(t_1,t_2]\times\R^{\cL}\to\R$ if for every test function $(\varphi, C_{r,\t}(x,t))$, with $C_{r,\t}(x,t) \ss \W\times(t_1,t_2]$, $I_i\varphi$ converges to $I\varphi$ locally uniformly in $C_{r,\t}(x,t)$.

\end{definition}

\begin{definition}[$\G$-convergence]
Consider $\{u_i\}_{i\in\N} \ss LSC(\W\times(t_1,t_2]) \cap LSC((t_1,t_2]\mapsto L^1(\w_\s))$. We say $u_n$ $\G$-converges to a function $u$ if:
\begin{enumerate}
\item For every sequence $(x_i,t_i) \to (x,t^-) \in \W\times(t_1,t_2]$,
\begin{align*}
\liminf_{i\to\8}u_i(x_i,t_i) \geq u(x,t),
\end{align*}
\item For every sequence $t_i \to t^- \in (t_1,t_2]$, $\|(u(\cdot,t) - u(\cdot,t_i))^+\|_{L^1(\w_\s)} \to 0$,
\item For every $(x,t) \in \W\times(t_1,t_2]$, there exists a sequence $(x_i,t_i) \to (x,t^-)$ such that $u_i(x_i,t_i)\to u(x,t)$,
\end{enumerate}
\end{definition}

If $u_i \to u$ in the $\G$-sense in $\W\times(t_1,t_2]$ and $u$ has a local minimum at some $(x,t)\in\W\times(t_1,t_2]$ then there exists a sequence $(x_i,t_i)\to(x,t^-)$ such that $u_i$ has a local minimum at $(x_i,t_i)$. We can use this last property whenever we are given a test function $\varphi$, strictly touching $u$ from below at $(x,t)$ in $C_{r,\t}(x,t)$. Then by a vertical translation we get test functions $(\varphi + d_i)$ touching $u_i$ from below at $(x_i,t_i)$ in $C_{r,\t}(x,t)$ such that $d_i\to0$ and $(x_i,t_i)\to(x,t^-)$. The following stability result uses mainly this idea.

\begin{theorem}[Stability]\label{thm:stability}
Let $I_i$, $I$ be a lower semicontinuous elliptic operator and $\{u_i\}_{i\geq 1}$, $u$ and $\{f_i\}_{i\geq 1}$, $f$ be sequences of functions such that:
\begin{enumerate}
\item $(u_i)_t - I_iu_i \geq f_i$ in the viscosity sense in $\W\times(t_1,t_2]$,
\item $I_i \to I$ weakly in $\W\times(t_1,t_2]$,
\item $u_i \to u$ in the $\G$-sense in $\W\times(t_1,t_2]$,
\item $\liminf_{i\to\8} f_i(x_i,t_i) \geq f(x,t)$ for every $(x_i,t_i)\to(x,t^-)$ in $\W\times(t_1,t_2]$,
\end{enumerate}
Then
\begin{align*}
u_t - Iu \geq f \text{ in $\W\times(t_1,t_2]$}.
\end{align*}
\end{theorem}

%%% Comparison Principle %%%

\subsection{Comparison Principle}

The Comparison Principle for elliptic equations states that whenever $u$ and $v$ are sub and super solutions of the same equation such that $u \leq v$ in the parabolic boundary of the domain then the order gets also preserved inside the domain. This implies immediately the uniqueness of solutions.

%After this, the Comparison Principle would be immediate if we can use the uniform ellipticity identity for viscosity solutions. However, this is not directly implied from the definitions. We prove that this identity holds for translation invariant operators by using the classical sup-convolution which regularizes the solution meanwhile preserving the equation.

% Maximum Principle %

\subsubsection{Maximum Principle}

\begin{theorem}[Maximum Principle]\label{thm:maximum_principle}
Let $I$ be a uniformly elliptic operator and $w$ a function such that,
\begin{alignat*}{2}
w_t - Iw &\leq f &&\text{ in $\W\times(t_1,t_2]$}.
\end{alignat*}
Then
\begin{align*}
\sup_{\W\times(t_1,t_2]} w \leq \sup_{\p_p(\W\times(t_1,t_2])} w + C\|(f-I0)^+\|_\8,
\end{align*}
for some universal constant $C>0$ depending on $\W$ but independent of $\s \in [1,2)$.
\end{theorem}

\begin{proof}
Assume without loss of generality that $\sup_{\p_p(\W\times(t_1,t_2])} w = 0$ and $I0 = 0$. Otherwise, apply the following proof to $(w - \sup_{\p_p(\W\times(t_1,t_2]} w)$ and $f - I0$ (recall also that from our definition of viscosity sub solution $\sup_{\W\times(t_1,t_2]} u < \8$).

We will use a rescaling of $\psi(y) = (2-|y|^2)\chi_{B_1}(y)$ as a test function for $w$. The important thing to notice is that $\cM^+_{\cL_0}\psi \leq -\d_0$ in some ball $B_{\d_1}$ for some universal constants $\d_0 > 0$ and $\d_1 \in (0,1)$ independent of $\s \in [1,2)$. It can be proved by using Proposition \ref{pro:Pucci_sigma_to_two} because, for each $\s \in [1,2)$, $\cM^+_{\cL_0}\psi$ is strictly negative in a neighborhood of the origin and neither this negative quantity or the neighborhood degenerate in the limit when $\s$ goes to two.

Assume that $M := \sup_{\W\times(t_1,t_2]} u \geq 0$ (otherwise there is nothing to prove) and that $\W \ss B_R$, for $R=\diam(\W)$. Under these assumptions we have $\varphi := \a\psi(\cdot/R)$ is a test function touching $u$ from above at $(x,t) \in \W\times(t_1,t_2]$ for some $\a \in [M/2,M]$. Then we have that
\begin{align*}
\|f^+\|_\8 \geq \varphi_t(x,t) - I\varphi(x,t) &\geq - \cM^+_{\cL_0}\varphi(x) \geq (M/2)(R\d_1^{-1})^{-\s}\d_0,
\end{align*}
giving us the desired bound.
\end{proof}

%As a consequence we obtain uniqueness of the zero solution for,
%\begin{alignat*}{2}
%u_t - Iu &= I0 &&\text{ in viscosity in $\W\times(t_1,t_2]$},\\
%u &= 0 &&\text{ in $\p_p(\W\times(t_1,t_2])$}.
%\end{alignat*}
%
%Another Corollary allows us to get a bound for $w$ only assuming that $w$ is bounded around $\W\times(t_1,t_2]$. It is proved by truncating $u$ away from the domain $\W\times(t_1,t_2]$.
%
%\begin{corollary}\label{cor:maximum_principle}
%Let $I$ be a uniformly elliptic operator and $w$ a function such that,
%\begin{alignat*}{2}
%w_t - Iw &\leq f &&\text{ in viscosity in $\W\times(t_1,t_2]$}.
%\end{alignat*}
%Then, for $\W \ss B_{R/2}$,
%\begin{align*}
%\sup_{\W\times(t_1,t_2]} w \leq \sup_{A_{R,a,b}} w + C\1\|(f-I0)^+\|_\8 + \|w^+\chi_{B_R^c}\|_{L^\8((t_1,t_2]\mapsto L^1(\w_\s))}\2,
%\end{align*}
%where $A_{R,a,b}=C_{R,b-a}(0,b) \sm (\W\times(t_1,t_2])$, $C>0$ depending on $\W$ but independent of $\s \in [1,2)$ and $R$.
%\end{corollary}
%
%Notice that as $R$ becomes larger and larger, $\sup_{C_{R,b-a}(0,b) \sm (\W\times(t_1,t_2])} w$ becomes larger and $\|w^+\chi_{B_R^c}\|_{L^\8((t_1,t_2]\mapsto L^1(\w_\s))}$ becomes smaller. In some cases it is possible to find an optimal $R$ for the previous estimate. For example, if $\W\times(t_1,t_2] = C_{1,1}$ and $w\chi_{B_1^c} \leq M|y|^{n+\s_0}$ for some $\s_0 < \s$ then $\sup_{C_{1,1}} w \leq C((\s-\s_0)^{-1}M + \|(f-I0)^+\|_\8)$.

% Equation for the difference %

\subsubsection{Uniform Ellipticity Identity for Viscosity Solutions}

From here the Comparison Principle would be immediate if we can use the uniform ellipticity identity for viscosity solutions. However, this is not directly implied from the definitions. We prove that this identity holds for translation invariant operators by using the classical sup-convolution which regularizes the solution meanwhile preserving the equation.

\begin{theorem}[Uniform ellipticity identity for viscosity solutions]\label{thm:equation_for_the_diference}
Let $I$ be a translation invariant, uniformly elliptic operator with respect to $\cL$, $f,-g\in USC(\W\times(t_1,t_2])$ and $u$ and $v$ such that, in the viscosity sense,
\begin{align*}
u_t - Iu &\leq f \text{ in $\W\times(t_1,t_2]$},\\
v_t - Iv &\geq g \text{ in $\W\times(t_1,t_2]$}.
\end{align*}
Then for $w = u-v$ the following holds also in viscosity,
\begin{align*}
w_t - \cM^+_{\cL}w \leq f-g \text{ in $\W\times(t_1,t_2]$}.
\end{align*}
\end{theorem}

\begin{definition}[Sup-convolution]
Let $t_1' \in (t_1,t_2)$ and $u \in USC(\R^n\times(t_1,t_2]) \cap L^\8(\R^n\times(t_1,t_2])$. We define the upper $\e$-envelope $u^\e:\R^n\times(t_1',t_2]\to\R$ as
\begin{align*}
u^\e(x,t) = \sup_{(y,s)\in\R^n\times[t_1',t]}\1u(y,s) - \e^{-1}P(y-x,s-t)\2.
\end{align*}
For $P(y,s) = (|y|^2-s)$.

Similarly we define the lower $\e$-envelope for $v \in LSC(\R^n\times(t_1,t_2]) \cap L^\8(\R^n\times(t_1,t_2])$ by $v_\e = -(-v)^\e$.
\end{definition}

The following properties can be proved by duality arguments as in \cite{Caffarelli95}.

\begin{property}
Let $u^\e$ be the upper $\e$-envelope for $u$ and for $(x,t) \in \R^n\times(t_1',t_2]$ let $(x^*,t^*) \in \R^n\times[t_1',t]$ such that
\begin{align*}
P(x^*-x,t^*-t) = \e(u(x^*,t^*)-u^\e(x,t)).
\end{align*}
Then,
\begin{enumerate}
\item $-u^\e \nearrow -u$ in the $\G$-sense as $\e\to0$.
\item $u^\e$ is $C^{1,1}(\R^n\times(t_1',t_2])$ from below in the parabolic sense, meaning that for every $(x,t) \in \R^n\times(t_1',t_2]$ the paraboloid $u(x^*,t^*) - \e^{-1}P^*(\cdot-x^*,\cdot-t^*)$, for $P^*(y,s) = (|y|^2+s)$, touches $u^\e$ from below and towards the past at $(x,t)$, that is
\begin{align*}
u^\e(x,t) &= u(x^*,t^*) - \e^{-1}P^*(x-x^*,x-t^*),\\
u^\e(y,s) &\geq u(x^*,t^*) - \e^{-1}P^*(y-x^*,s-t^*) \text{ for $(y,s) \in \R^n\times(t_1',t]$}. 
\end{align*}
\end{enumerate}
\end{property}

\begin{remark}\label{rmk:regularity_sup_convolution}
The last property tells us that for every $t \in (t_1',t_2]$, $u^\e(\cdot,t)$ is semi-convex. By a result of Alexandroff we get that it is twice differentiable a.e. and by Property \ref{pro:viscosity_solution}, the non-local operators can be evaluated at the same points. Regarding the regularity in time, we know that for every $x \in \R^n$, the function $s \mapsto u^\e(x,s) + \e s$ is nondecreasing, which implies that $u^\e_t(x,\cdot)$ is also well defined almost everywhere. 
\end{remark}

In a similar way in which the Stability Theorem \ref{thm:stability} is treated the following lemma can be deduced.

\begin{lemma}\label{lem:equation_sup_convolution}
Let $I$ be a translation invariant, uniformly elliptic operator with respect to $\cL$, $f\in USC(\W\times(t_1,t_2])$ and suppose $u \in L^\8(\R^n \times(t_1,t_2])$ satisfies in viscosity,
\begin{align*}
u_t - Iu &\leq f \text{ in $\W\times(t_1,t_2]$}.
\end{align*}
Then, for $\W' \cc \W$, $t_1'' \in (t_1',t_2)$ and $\e$ small enough $u^\e$ also satisfies in viscosity
\begin{align*}
u^\e_t - Iu^\e \leq f + \w(\e) \text{ in $\W'\times(t_1'',t_2]$},
\end{align*}
for some $\w(\e)\to0$ as $\e\to0$.
\end{lemma}

The following result is a relaxed version of Theorem \ref{thm:equation_for_the_diference}. Notice that $u$ and $v$ have a semicontinuity hypothesis for their tails.

\begin{lemma}
Let $I$ be a translation invariant, uniformly elliptic operator with respect to $\cL$, $f,-g\in USC(\W\times(t_1,t_2])$, $u,-v \in USC(\R^n\times(t_1,t_2]) \cap L^\8(\R^n\times(t_1,t_2])$ functions such that, in the viscosity sense,
\begin{align*}
u_t - Iu &\leq f \text{ in $\W\times(t_1,t_2]$},\\
v_t - Iv &\geq g \text{ in $\W\times(t_1,t_2]$}.
\end{align*}
Then for $w = u-v$ the following holds also in viscosity,
\begin{align*}
w_t - \cM^+_{\cL} w \leq f-g \text{ in $\W\times(t_1,t_2]$}.
\end{align*}
\end{lemma}

%In \cite{Caffarelli09} the proof of the analogous Lemma was simplified by the Lemma 4.3, analogous to the last property we have in \ref{pro:viscosity_solution}. In this case we do not have control any more of the set where the time derivative can be evaluated for $w^\e := (u^\e - v_\e)$ so we have to go back to the classical proof using the parabolic convex envelope or Jensen's Lemma.

\begin{proof}
We will show that for $w^\e := (u^\e - v_\e)$,
\begin{align*}
w^\e_t - \cM^+_{\cL} w^\e \leq f-g + \w(\e) \text{ in $\W\times(t_1,t_2]$},
\end{align*}
for some $\w(\e)\to0$ as $\e\to0$. The result for $w$ then follows from the Stability Theorem \ref{thm:stability}.

Let $(\varphi,C_{r,\t}(x,t))$ a test function strictly touching $w^\e$ from above at $(x,t) \in \W\times(t_1,t_2]$. We need to show that,
\begin{align*}
\varphi_{t^-}(x,t) - \cM^+_{\cL} \varphi(x,t) \leq (f-g)(x,t) + \w(\e).
\end{align*}

Fix $\k \in (0,1)$. By subtracting a small number $\d>0$ we have that $\psi_\k = (\varphi - w^\e - \d)$ still satisfies $\psi_\k \geq 0$ outside $C_{\k r,\k \t}(x,t)$ but now has a strictly negative minimum $-\d$ at $(x,t)$. Let
\begin{align*}
\Sigma_\k = &\{(y,s) \in C_{\k r,\k \t}(x,t): \exists p \in \R^n \text{ such that }\\
&\psi_\k(z,\varsigma) \geq p\cdot(z-y) + \psi_\k(y,s) \ \forall (z,\varsigma)\in C_{\k r,s-\k \t}(x,s)\}.
\end{align*}
We know by the construction of the parabolic convex envelope due to K. Tso \cite{Tso85} that $\Sigma_\k$ has positive measure (depending on $\e$).

Recall now Remark \ref{rmk:regularity_sup_convolution} which says that $u^\e, v_\e \in C^{1,1}_xC^1_t(y,s)$ a.e. $(y,s) \in C_{\k r,\k \t}(x,t)$. We obtain in this way $(x_\k,t_\k) \in \Sigma_\k$ such that $u^\e, v_\e$ both are in $C^{1,1}_xC^1_t(x_\k,t_\k)$. Given that $\psi_\k$ has a supporting plane from below at $(x_\k,t_\k)$ we get that,
\begin{align*}
(\psi_\k)_{t^-}(x_\k,t_\k) - \cM^+_{\cL} \psi_\k(x_\k,t_\k) \geq 0.
\end{align*}
By uniform ellipticity and Lemma \ref{lem:equation_sup_convolution},
\begin{align*}
\varphi_t(x_\k,t_\k) - \cM^+_{\cL} \varphi(x_\k,t_\k) &\leq w^\e_t(x_\k,t_\k) - \cM^+_{\cL} w^\e(x_\k,t_\k)\\
&\leq \1u^\e_t(x_\k,t_\k) - \cM^+_{\cL} u^\e(x_\k,t_\k)\2\\
&{} - \1(v_\e)_t(x_\k,t_\k) - \cM^+_{\cL} v_\e(x_\k,t_\k)\2\\
&\leq (f-g)(x_\k,t_\k) + \w(\e).
\end{align*}
We conclude by sending $\k \to0$.
\end{proof}

The proof of Theorem \ref{thm:equation_for_the_diference} can be now recovered by mollifying and truncating $u$ and $v$ outside of $\W\times(t_1,t_2]$. We omit its proof as it is mainly technical. As a consequence we finally obtain the following Comparison Principle.

\begin{theorem}\label{thm:comparison_principle}
Let $I$ be a translation invariant, uniformly elliptic operator, $f,-g\in USC(\W\times(t_1,t_2])$ and $u$ and $v$ functions such that in viscosity,
\begin{align*}
u_t - Iu &\leq f \text{ in $\W\times(t_1,t_2]$},\\
v_t - Iv &\geq g \text{ in $\W\times(t_1,t_2]$}.
\end{align*}
Then for $w = u-v$,
\begin{align*}
\sup_{\W\times(t_1,t_2]} w \leq \sup_{\p_p(\W\times(t_1,t_2])} w + C\|(f-g)^+\|_\8,
\end{align*}
for some universal constant $C>0$ depending on $\W$ but independent of $\s \in [1,2)$.
\end{theorem}

\begin{remark}
The translation invariance of $I$ is used in a crucial way in the proofs above, however, the same proof can be carried out if we can control the error $I(x+x_0,t+t_0,u(\cdot-x_0,\cdot-t_0)) - I(x,t,u)$ uniformly with respect to $u$. For instance, given that $I$ is translation invariant and $\varphi$ is smooth/integrable, we can consider,
\begin{align*}
\tilde I(x,t,u) := I(x,t,u+\varphi).
\end{align*}
which is not necessarily translation invariant but satisfies,
\begin{align*}
&|\tilde I(x+x_0,t+t_0,u(\cdot-x_0,\cdot-t_0)) - \tilde I(x,t,u)|\\
\leq &\sup_{L\in\cL} |L\varphi(x-x_0,t-t_0)-L\varphi(x,t)|.
\end{align*}
Then Lemma \ref{lem:equation_sup_convolution} is still valid with $\w$ depending on $\varphi$ which implies the comparison principle for $\tilde I$.
\end{remark}

%%% Existence and Uniqueness of Viscosity Solutions %%%

\subsection{Existence and Uniqueness of Viscosity Solutions}\label{sec:existence}

We proced now to construct some barriers. Together with the Perron's method they imply the existence of viscosity solutions taking the prescribed boundary values in a continuous way.

%% Barriers %

\subsubsection{Barriers}

\begin{lemma}\label{lem:barrier_boundary}
For $\s \in [1,2)$, there exists a non negative function $\psi:\R^n\times(-\8,0]\to[0,1]$ such that for some universal $\k,r_0 > 0$ independent of $\s$, and
\begin{align*}
A = \{(y,s) \in B_{1+r_0}\times(-2\k^{-1},0]: 1 < |y| \leq 1 + (\k/2)r_0(2\k^{-1}-s)\}.
\end{align*}
\begin{alignat*}{2}
\psi_t-\cM^+_{\cL_0}\psi &> \k/2 &&\text{ in $A$},\\
\psi &= 0 &&\text{ in $B_1\times\{0\}$},\\
\psi &= 1 &&\text{ in $(\R^n\times(-\8,0])\sm A$},
\end{alignat*}
\end{lemma}

\begin{proof}
Let
\begin{align*}
\varphi(y) = ((|y|-1)^+)^\a.
\end{align*}
We will show first that, for $\a,r_0 \in (0,1)$ sufficiently small, $\cM^+_{\cL_0}\varphi < -\k$ in $\bar B_{1+r_0}\sm B_1$ for some universal $\k > 0$.

By radial symmetry it is enough to show the identity for $x = (1+r)e_1$ with $r\in(0,r_0])$. Let $x_0 = (1+r_0)e_1$, by scaling the graph of $\varphi$, centered at $e_1$ and sending $(x,\varphi(x))$ to $(x_0,\varphi(x_0))$ we also see that we can reduce the computation to $x_0$. To be more specific let $\r = r/r_0 \in (0,1]$ and,
\begin{align*}
\tilde \varphi(y) = \r^{-\a}\varphi\1\r(y-e_1) + e_1\2,
\end{align*}
which satisfies $\tilde\varphi \leq \varphi$ and $\tilde\varphi((1+|y|)e_1) = \varphi((1+|y|)e_1)$.

Given $(K,b) \in \cL_0$, define
\begin{align*}
\tilde K(y) &= K\1\r y\2 \in \cK_0',\\
\tilde b &= \r^{\s-1}\1b + (2-\s)\int_{B_1\sm B_\r} \frac{yK(y)}{|y|^{n+\s}}dy\2,
\end{align*}
then,
\begin{align*}
L_{K,b} \varphi(x) = \r^{\a-\s}L_{\tilde K, \tilde b}\tilde\varphi(x_0) \leq \r^{\a-\s}L_{\tilde K, \tilde b}\varphi(x_0) \leq \r^{\a-\s}\cM^+_{\cL_0}\varphi(x_0).
\end{align*}
By taking the supremum on the left hand side over $(K,b) \in \cL_0$ we conclude that $\cM^+_{\cL_0}\varphi(x_0) < -\k$ implies $\cM^+_{\cL_0}\varphi(x) < -\k$.

Letting $\s\nearrow2$ and $\a\in(0,1/2)$, we recover
\begin{align*}
\cM^+_{\cL_0}\varphi(x_0) &\to \int_{\p B_1} \1\1\theta^t D^2 \varphi(x_0)\theta\2^+ \l - \1\theta^t D^2 \varphi(x_0)\theta\2^- \L \2d\theta + \b|D\varphi(x_0)|,\\
&\leq \sup_{A \in [\l/C,C\L]} \trace(AD^2\varphi(x_0)) + \b|D\varphi(x_0)|,\\
&= \a r_0^{\a-2}\1\frac{\l(\a-1)}{C} + \1C\L\frac{(n-1)}{r_0+1} + \b\2r_0\2,\\
&< \a r_0^{\a-2}\1-\frac{\l}{2C} + \1C\L(n-1) + \b\2r_0\2.
\end{align*}
The previous quantity can be made smaller than
\begin{align*}
-\k = \frac{\a r_0^{\a-2}\l}{4C} < 0,
\end{align*}
for $r_0$ sufficiently small, independently of how small is $\a$. By Proposition \ref{pro:Pucci_sigma_to_two}, we get that $\cM^+_{\cL_0}\varphi(x_0) < -\k$ for $\s \in [\s_0,2)$ close to two.

As $\a\searrow0$, $\varphi\to\chi_{B_1^c}$ for which, $\cM^+_{\cL_0}\chi_{B_1^c}(x_0)\to-\8$ as $r_0 \searrow 0$ uniformly for $\s\in[1,\s_0)$ away from two. For $r_0$ sufficiently small we have that $\cM^+_{\cL_0}\chi_{B_1^c}(x_0) < -\k$. Finally we fix $\a\in(0,1/2)$ sufficiently small such that also $\cM^+_{\cL_0}\varphi(x_0) < -\k$ holds for $\s\in[1,\s_0)$.

Now that we have proven that $\cM_{\cL_0}\varphi < -\k$ we define
\begin{align*}
\psi(y,s) = \max\1\varphi(y) - \frac{\k}{2} s,1\2,
\end{align*}
which is the desired barrier.
\end{proof}

By combining the previous lemma with the Comparison Principle we obtain the following corollary.

\begin{corollary}\label{cor:boundary_data}
Let $\e,\d_x,\d_t \in (0,1)$, $C_0,C_{1,1}\geq 0$, $a>0$ and $u$ such that,
\begin{enumerate}
\item $\W \ss B_R \sm B_{\d_x/2}^c((\d_x/2)e_1)$,
\item $u_t - \cM^+_{\cL_0}u \leq C_0$ in $\W \times (-a,0]$,
\item $u(0,0) = 0$
\item $u \leq \e$ in $C_{\d_x,\d_t} \cap \p_p(\W \times (-\d_t,0])$,
\item $u \leq C_{1,1}$ in $\p_p(\W \times (-\d_t,0])$.
\end{enumerate}
Then, for $\k,r_0$ and $\psi$ as in Lemma \ref{lem:barrier_boundary} and $\theta = \min\1\frac{\d_x}{2+r_0}, (\k\d_t)^{1/\s}\2$, we have
\begin{align*}
u(y,s) \leq \e + \1\frac{C_0\k}{2}+C_{1,1}\2\theta^{\s}\psi\1\frac{y-\theta e_1}{\theta},\frac{s}{\theta^\s}\2 \text{ for $s \in (-a,0]$}.
\end{align*}
\end{corollary}

For the initial values we can use a much simpler barrier. Consider $\b:\R^n\to[0,1]$ a smooth function such that $\b = 0$ in $B_1$ and $\b = 1$ in $B_2^c$. We know that $\cM^+_{\cL_0}\b$ is globally bounded and then,
\begin{align*}
\psi(y,s) = \b(y) + (1+\|\cM^+_{\cL_0}\b\|_\8)s,
\end{align*}
satisfies,
\begin{alignat*}{2}
\psi_t - \cM_{\cL_0}\psi &\geq 1 &&\text{ in $\R^n\times\R$},\\
\psi &= 0 &&\text{ in $B_1\times\{0\}$},\\
\psi &\geq 1 &&\text{ in $B_2^c\times(-\8,0]$}.
\end{alignat*}

As a corollary of the Comparison Principle we obtain.

\begin{corollary}\label{cor:initial_data}
Let $\e,\d_x,\d_t \in (0,1)$, $C_0,C_{1,1}\geq 0$, $a>0$ and $u$ such that,
\begin{enumerate}
\item $0 \in \W$,
\item $u_t - \cM^+_{\cL_0}u \leq C_0$ in $\W \times (0,a]$,
\item $u(0,0) = 0$
\item $u \leq \e$ in $\bar C_{\d_x,\d_t}(0,\d_t) \cap \p_p(\W \times (0,a])$,
\item $u \leq C_{1,1}$ in $\p_p(\W \times (0,a])$.
\end{enumerate}
Then for $\theta = \min\1\frac{\d_x}{2},\d_t^{1/\s}\2$ and $\psi$ as defined after Corollary \ref{cor:boundary_data}, we have
\begin{align*}
u(y,s) \leq \e + \1C_0+C_{1,1}\2\theta^{\s}\psi\1\frac{y}{\theta},\frac{s}{\theta^\s}\2 \text{ for $s\in[0,a]$.}
\end{align*}
\end{corollary}

% Perron's Method %

\subsubsection{Perron's Method}

\begin{theorem}\label{thm:existence_and_uniqueness}
Let
\begin{enumerate}
\item $\W$ a bounded domain satisfying the exterior ball condition,
\item $I$ be a translation invariant, uniformly elliptic operator.
\item $f \in C(\W\times(t_1,t_2]) \cap L^\8(\bar\W\times[a,b])$,
\item $g \in C((t_1,t_2]\to L^1(\w_\s)) \cap L^\8(C_{R/2,b-a}(0,b))$ continuous at $\bar\W\times[a,b]$.
\end{enumerate}
Then, the Dirichlet problem,
\begin{alignat*}{2}
u_t - Iu &= f &&\text{ in $\W\times(t_1,t_2]$},\\
u &= g &&\text{ in $\p_p(\W\times(t_1,t_2])$},
\end{alignat*}
has a unique viscosity solution taking the boundary and initial values in a continuous way.
\end{theorem}

\begin{proof}
The uniqueness part follows from the Comparison Principle \ref{thm:comparison_principle}. For the existence we use the Stability Theorem \ref{thm:stability} and the Comparison Principle \ref{thm:comparison_principle}. By the standard Perron's method we can show the existence of a viscosity solution $u$, defined as the smallest viscosity super solution above the boundary values given by $g$. The Dirichlet boundary problem gets solved by $u$ provided that there exists barriers that force $u$ to take the boundary and initial values in a continuous way. This is implied by Corollaries \ref{cor:boundary_data} and \ref{cor:initial_data}.
\end{proof}

%%% Alexandroff-Bakelman-Pucci type of estimate %%%

\section{Alexandroff-Bakelman-Pucci type of estimate}\label{section:Alexandroff-Bakelman-Pucci type of estimate}

We prove a non-local version of the classical ABP estimate. These type of estimates play an important role in regularity theory, since they allow to pass an estimate in measure to a pointwise estimate. This turn out to be crucial to prove the Point estimate, the Oscillation Lemma and the Harnack Inequality in the coming sections. For the local case, we refer to the work of K. Tso \cite{Tso85} and L. Wang \cite{Wang92}.

\subsection{Weak Point Estimate}

The following result is a modification of a Lemma established by L. Silvestre in \cite{Silvestre11} where it bounds the distribution of the solution in a way that resembles the mean value theorem. Actually, these linear non-local operators somehow have this formula built-in in their own definition but the price is that the estimate degenerates as $\s\nearrow 2$. Here we consider the different distributions in dyadic rings as was done in \cite{Caffarelli09} for the elliptic case. This is the first step to get a uniform control with respect to $\s$.

\begin{lemma}[Key Lemma]\label{lem:key_lemma}
Let $\D t \in (0,1]$ and suppose $u \geq 0$ satisfies,
\begin{alignat*}{2}
u_t - \cM_{\cL_0}^-u &\geq -f(t) &&\text{ in $C_{1,\D t}$},\\
\|f^+\|_{L^1(-\D t, 0])} &\leq \D t.
\end{alignat*}
Then,
\begin{align*}
\inf_{C_{1/2,\D t/2}} u \geq \D t,
\end{align*}
provided that for some $M > 0$,
\begin{align*}
\frac{|\{u > M2^{2i}\} \cap (B_{2^{i+1}}\sm B_{2^i}) \times (-\D t,-\D t/2]|}{ |(B_{2^{i+1}}\sm B_{2^i}) \times (-\D t,-\D t/2]|} \geq M^{-1},
\end{align*}
for each $i \in \{0,1,\ldots,(k-1)\}$ where $k(2-\s) \geq C$ for some universal constant $C$ independent of $M$ and $\s \in [1,2)$.
\end{lemma}

\begin{proof}
It suffices to provide a sub solution of the same equation that remains below $u$ on $\p_p C_{3/4,\D t}$ and grows at least up to $\D t$ everywhere in $C_{1/2,\D t/2}$. The following ansatz uses the values given by $u$ in $B_1^c$ allowing some growth for the barrier about the origin:
\begin{align*}
 v(x,t) &:= \1m(t)\varphi(x) - \int_{-\D t}^t f^+(s)ds\2\chi_{B_1}(x) + u(x,t)\chi_{B_1^c}(x),
\end{align*}
where $\varphi$ is a smooth function taking values between zero and one with
\begin{align*}
\supp \varphi = B_{3/4},\\
\varphi = 1 \text{ in $B_{1/2}$}
\end{align*}
and $m$ is function such that $m(-\D t) = 0$. Notice that $m$ determines the grow of $v$ in $C_{1/2,\D t}$. To prove the lemma we will have to arrange $m$ such that:
\begin{enumerate}
\item $v_t - \cM^-_{\cL_0}v \leq -f$ in $C_{3/4,\D t}$,
\item $m \geq 2\D t$ in $[-\D t/2,0]$.
\end{enumerate}

We estimate $v_t - \cM^-_{\cL_0}v$  in $C_{3/4,\D t}$ by the uniform ellipticity identity,
\begin{align*}
v_t - \cM^-_{\cL_0}v \leq m'\varphi - f^+ - m\cM^-_{\cL_0}\varphi - \cM^-_{\cL_0}(u\chi_{B_1^c}).
\end{align*}
Note that $u\chi_{B_1^c}$ is smooth in $C_{3/4,\D t}$ so we can estimate $\cM^-_{\cL_0}(u\chi_{B_1^c})$ in this region in terms of the sets appearing in the hypothesis of the lemma,
\begin{align*}
\cM^-_{\cL_0}(u\chi_{B_1^c})(x,t) &\geq \l(2-\s) \sum_{i=0}^{k-1} \int_{y + x \in B_{2^{i+1}}\sm B_{2^i}} \frac{u(y+x,t)}{|y|^{n+\s}}dy,\\
 &\geq c(2-\s)M\sum_{i=0}^{k-1}2^{(2-\s)i}\frac{|G_i(t)|}{|B_{2^{i+1}}\sm B_{2^i}|},
\end{align*}
where
\begin{align*}
G_i(t) := \{y \in B_{2^{i+1}}\sm B_{2^i} : u(y,t) > M2^{2i}\}.
\end{align*}
So, in order to get that $v$ is a sub solution it suffices that,
\begin{align}\label{eq:key_lemma}
m'\varphi - m\cM^-_{\cL_0}\varphi &\leq F(t)\\
\nonumber &:= c(2-\s)M\sum_{i=0}^{k-1}2^{(2-\s)i}\frac{|G_i(t)|}{|B_{2^{i+1}}\sm B_{2^i}|}
\end{align}

This is the moment to fix $m$. The previous computation suggests us to take $m$ as the solution of an ordinary differential equation:
\begin{align*}
\begin{cases}
m' + am = F \quad \text{$(a>0)$},\\
m(-\D t) = 0.
\end{cases}
\end{align*}
Notice that $\cM^-_{\cL_0}\varphi \geq 0$ if $\varphi \leq \d$ for some universal $\d>0$. In that case the equation automatically implies \eqref{eq:key_lemma} as $\varphi \leq 1$. On the other hand, if $\varphi > \d$, we can also imply \eqref{eq:key_lemma} by taking $a = \|(\cM^-_{\cL_0}\varphi)^-\|_\8/\d$.

We finally need to check that we can make $m \geq 2\D t$ in $[-\D t/2,0]$. By integrating the previous equation, we get
\begin{align*}
m(t) &= \int_{-\D t}^t F(s)e^{-a(t-s)}ds,\\
&= c(2-\s)M\sum_{i=0}^{k-1}2^{(2-\s)i}\int_{-\D t}^t\frac{|G_i(s)|}{|B_{2^{i+1}}\sm B_{2^i}|}e^{-a(t-s)}ds.
\end{align*}
By the hypothesis of the lemma we get that, for $t \geq -\D t/2$,
\begin{align*}
 m(t)\geq c(2-\s)\frac{2^{(2-\s)k}-1}{2^{2-\s}-1}e^{-a\D t}\D t \geq c(2^{(2-\s)k}-1)\D t,
\end{align*}
which is larger than $2\D t$, independently of $\s \in[1,2)$, provided that $(2-\s)k$ is sufficiently large.
\end{proof}

\begin{corollary}\label{cor:key_corollary}
Let $k \sim (2-\s)^{-1}$, as required for the conclusion in the previous lemma, $r \in (0,1]$, $\D t \in (0,(2^{-k}r)^\s]$ and let $u \geq 0$ satisfy,
\begin{align*}
u_t - \cM_{\cL_0}^-u &\geq -C_0 \text{ in $C_{2^{-k}r,\D t}$},\\
u(0,0) &= 0.
\end{align*}
Then, given $M>0$, there exists some non negative integer $i \leq (k-1)$ such that for $r_i = 2^{-i}r$,
\begin{align*}
\frac{|\{u > MC_0r^{-(2-\s)}r_i^2\} \cap (B_{r_i}\sm B_{r_i/2}) \times (-\D t,-\D t/2]|}{ |(B_{r_i}\sm B_{r_i/2}) \times (-\D t,-\D t/2]|} < M^{-1}.
\end{align*}
\end{corollary}

The following Corollary follows from the proof of the Lemma \ref{lem:key_lemma}. It is equivalent to the estimate in \cite{Silvestre11} and it gives a point estimate in $L^1$ for the solution, with the draw back that the estimate degenerates as $\s$ goes to two. Later on we will see how to obtain an estimate independent of $\s$ by allowing the estimate to depend now on the weak $L^\e$ norm of the solution for $\e$ sufficiently small.

There is one technical observation, the (spatial) integral on the left-hand side is computed in $\R^n$ instead of an annular region. This can be retrieved in two steps, first bounding the integral in $(\R^n \sm B_{1/2})\times(-1,-1/4]$ and then using a point in $B_{7/8} \sm B_{5/8} \times (-1/2,-1/4]$ where $u$ is bounded by the previous step to complete the bound in the missing region.

\begin{corollary}[Weak Point Estimate]\label{cor:weak_point_estimate}
Let $u \geq 0$ satisfy,
\begin{align*}
u_t - \cM_{\cL_0}^-u &\geq -f(t) \text{ in $C_{1,1}$}.
\end{align*}
Then for some universal $C>0$ independent of $\s$,
\begin{align*}
(2-\s)\int_{-1}^{-1/2}\|u(t)\|_{L^1(\w_\s)}dt \leq C\1u(0,0) + \|f^+\|_{L^1(-1,0)}\2.
\end{align*}
\end{corollary}

%%% Initial constructions %%%

\subsection{Preliminary Lemmas}

Here we fix some hypothesis and notation that we will use for the next results.

\begin{enumerate}
\item $u_t - \cM^-_{\cL_0}u \geq -1$ in $C_{2,1}$.
\item $u \geq 1$ in $\p_p C_{1,1}$
\item $\sup_{C_{1,1}}u^- = |u(x_0,t_0)| \in (0,1]$ for some $(x_0,t_0) \in C_{1,1}$.
\item Let $\G$ be the parabolic convex envelope of $u$ supported in $B_d$ for some $d \geq 2$ sufficiently large and to be fixed,
\begin{align*}
\G(x,t) := &\sup\{p\cdot(x-x_0) + h:\\
&p\cdot(y-x_0) + h \leq -u^-(y,s) \ \forall (y,s) \in C_{d,1+t}(0,t)\}.
\end{align*}
\item Let $ D\G(x,t)$ be the set of sub differentials of $\G$ at $(x,t)$, 
\begin{align*}
 D\G(x,t) := \{p \in \R^n: p\cdot(y-x) + \G(x,t) \leq \G(y,s) \ \forall (y,s) \in C_{d,1+t}(0,t)\}.
\end{align*}
Note that $ D\G(B_d,t) =  D\G(B_1,t)$. We denote 
\begin{align*}
| D\G(x,t)| := \sup_{p\in D\G(x,t)} |p|.
\end{align*}
\item Let $h(\cdot,t): D\G(B_d,t)\to(-\8,0]$ be the Legendre transform of $\G$ centered at $x_0$ (minimum), that is
\begin{align*}
h(p,t) &:= \inf_{y \in B_d} \1\G(y,t) - p\cdot(y-x_0)\2\\
&= \sup\{h:p\cdot (y-x_0) + h \leq -u^-(y,t) \ \forall y \in B_d\}.
\end{align*}
\item Let $\Phi(x,t) := ( D\G(x,t),h( D\G(x,t),t))$.
\item Let $\Sigma := \{u = \G\} \ss C_{1,1}$ be the contact set.
\item Given $(p,h) \in \R^n\times\R$, let
\begin{align*}
P_{p,h}(y) = (p\cdot(y-x_0) + h)\chi_{B_2}(y) + \chi_{B_2^c}(y).
\end{align*}
\end{enumerate}

The following are some preliminary lemmas.

%%% Lemma %%%

\begin{lemma}\label{lem:MP_greater_than_zero}
Given $p\in D\G(B_1,t)$ and $h = h(p,t)$, the following properties hold:
\begin{enumerate}
\item $|p| \leq \frac{1}{d-1}$,
\item $-\frac{d+2}{d-1}\sup_{C_{1,1}}u^- \leq P_{p,h} \leq 0$ in $B_2$,
\item $\cM^-_{\cK_0} P_{p,h} > 0$ in $B_1$ provided that $d$ is sufficiently large, independently of $\s \in [1,2)$.
\end{enumerate}
\end{lemma}

\begin{proof}
(1) and (2) follow from the fact that the plane $y \mapsto (p\cdot(y-x_0) + h)$ remains below zero in $B_d$ and crosses the level set $-\sup_{C_{1,1}}u^- \in [-1,0)$ at some point in $B_1$. Using these two properties we can estimate $\cM^-_{\cK_0} P_{p,h}$ in $B_1$ in the following way. For $K \in \cK_0$ and $x\in B_1$, we have
\begin{align*}
L_KP_{p,h}(x) &= (2-\s)\int_{B_1^c} (P_{p,h}(y+x) - P_{p,h}(x))\frac{K(y)dy}{|y|^{n+\s}},\\
&\geq (2-\s)\1\int_{B_2^c(-x)}\frac{K(y)dy}{|y|^{n+\s}} + \int_{B_2(-x) \sm B_1}p\cdot(y-x)\frac{K(y)dy}{|y|^{n+\s}}\2,\\
&\geq (2-\s)\1C_{1,1} - \frac{C_2}{d-1}\2.
\end{align*}
This implies (3) by taking $d$ sufficiently large.
\end{proof}

%%% Lemma %%%

\begin{lemma}\label{lem:properties_h}
Given $(t,t+\D t] \ss (-1,0]$, the following properties hold for $h$:
\begin{enumerate}
\item The domain of $h(\cdot,t)$ is non decreasing in time. i.e.
\begin{align*}
 D\G(B_1,t) \ss  D\G(B_1,t+\D t).
\end{align*}
\item $h$ is non increasing in time.
\item $h$ is Lipschitz in time. More precisely, for $p \in  D\G(x,t)$
\begin{align*}
\D h := h(p,t+\D t) - h(p,t) \geq -C\D t,
\end{align*}
for some universal $C$.
\end{enumerate}
\end{lemma}

\begin{proof}
The first two properties are consequences of the monotonicity of $\G$. If at time $t$, the plane $y \mapsto (p\cdot(y-x_0) + h)$ is a supporting plane for the graph of $\G(\cdot,t)$ then at time $(t + \D t)$ it crosses or touches the graph of $\G(\cdot,t+\D t) \leq \G(\cdot,t)$ while remaining below zero in $B_d$. Therefore by lowering $h$ we can find a supporting plane for $\G(\cdot,t+\D t)$ with the same slope $p$.

For the last part we fix $p \in  D\G(B_1,t)$ and $h = h(p,t)$. Assume that $\D h < 0$ and consider the following test function,
\begin{align*}
v(y,s) &= P_{p,h}(y) + \frac{\D h}{2\D t}(s-t),
\end{align*}
and note that $v$ has to cross $\G$ in $B_2\times\{t+\D t\}$. By the definition of $\G$, $\1P_{p,h} + \frac{\D h}{2\D t}\2$ also has to cross $u$ in $C_{1,\D t}(0,t+\D t)$ meanwhile remaining below $u$ in $\p_pC_{1,\D t}(0,t+\D t)$. Let $t_1 \in (t,t+\D t]$ be the last time when $\1P_{p,h} + \frac{\D h}{2\D t}\2 < u$,
\begin{align*}
t_1 = \sup\3s \in (t,t+\D t]:\1P_{p,h} + \frac{\D h}{2\D t}\2(\cdot,s) < u(\cdot,s) \text{ in $C_{1,\D t}(0,t+\D t)$}\4
\end{align*}
Then $\tilde v(y,s) = v(y,s-(t+\D t - t_1))$ is a test function touching $u$ from below at some $(x_1,t_1) \in C_{1,\D t}(0,t)$. Plugging it into the equation for $u$ and using Lemma \ref{lem:MP_greater_than_zero} we obtain that,
\begin{align*}
-1 &\leq \tilde v_t(x_1,t_1) - \cM^-_{\cL_0}\tilde v(x_1,t_1) \leq \frac{\D h}{2\D t} + \b|p|,
\end{align*}
which concludes the proof as $|p|$ remains bounded according to Lemma \ref{lem:MP_greater_than_zero}.
\end{proof}

\begin{corollary}\label{cor:h}
Given $p \in  D\G(B_1,t)$ and $h = h(p,t)$, then 
\begin{align*}
P_{p,h} - C\D t\leq \G \text{ in $C_{1,1+t+\D t}(0,t+\D t)$}.
\end{align*}
\end{corollary}

The following lemma can be found in \cite{Davila12}.

\begin{lemma}\label{tapaarriba}
Let $\G:C_{r,\D t} \to \R$ be a parabolic convex function such that
\begin{align*}
\frac{|\{\G > M\} \cap (B_r \sm B_{r/2})\times(-\D t,-\D t/2]|}{|(B_r \sm B_{r/2})\times(-\D t,-\D t/2]|} < \e_0.
\end{align*} 
Then $\G \leq M$ in $C_{r/2,\D t/2}$ provided that $\e_0$ is sufficiently small, depending only on the dimension.
\end{lemma}

%\begin{proof}
%By the convexity and monotonicity of $\G$ we can assume that
%\begin{align*}N := \sup_{C_{r/2,\D t/2}} \G = \G(r/2e_1, -\D t/2) \leq \G \text{ in $A$},
%\end{align*}
%where
%\begin{align*}
%A = \{(x,t) \in (B_r\sm B_{r/2})\times(-\D t,-\D t/2]: x\cdot e_1 > r/2\}.
%\end{align*}
%Then we just need to fix $\e_0$ equal to $|A|/|(B_r\sm B_{r/2})\times[-\D t,-\D t/2]| \sim 1$, to obtain that the hypothesis of the Lemma implies $N \leq M$.
%\end{proof}

%%% Covering the Contact Set: Bakelman-Pucci type of estimate %%%

\subsection{Covering the Contact Set: Alexandroff-Bakelman-Pucci type of estimate}

We show in the next two lemmas how to cover the contact set $\Sigma$ with pieces where $u$ detaches from $\G$ in a controlled way. For this we use the Key Lemma \ref{lem:key_lemma} and the tools from the previous section. The first result finds a configuration for each point in $\Sigma$ meanwhile the second lemma provides an algorithm which produces a covering with some desired properties.

\begin{lemma}\label{lem:covering_piece}
Let $k \sim (2-\s)^{-1}$ as in Lemma \ref{lem:key_lemma}, $r \in (0,1]$, $\D t \in (0,(2^{-k}r)^2]$, $(x,t) \in \Sigma \cap C_{1,1-\D t}(0,-\D t/2)$, $p\in  D\G(x,t)$, $h = h(p,t)$. There exists some non negative integer $i \leq (k-1)$, such that the following holds for $r_i = 2^{-i}r$:
\begin{enumerate}
\item\textbf{ Control for $u$ detaching from $P_{p,h}$:} For some universal $C$ and $\e_0$ as in the Lemma \ref{tapaarriba},
\begin{align*}
\frac{|\{u > P_{p,h} + Cr^{-(2-\s)}r_i^2\} \cap R_{r_i,\D t}(x,t)|}{|R_{r_i,\D t}(x,t)|} < \e_0,
\end{align*}
where,
\begin{align*}
 R_{r_i,\D t}(x,t) := (B_{r_i}(x) \sm B_{r_i/2}(x)) \times (t-\D t,t-\D t/2].
\end{align*}
\item \textbf{Flatness for $\G$:} In $C_{r_i/2,\D t}(x,t+\D t)$
\begin{align*}
|\G - P_{p,h}| \leq Cr^{-(2-\s)}r_i^2.
\end{align*}
\item \textbf{Control of the jacobian measure of $\Phi$:}
\begin{align*}
\frac{|\Phi(C_{r_i/4,\D t}(x,t+\D t))|}{|C_{r_i/4,\D t}(x,t +\D t)|} \leq Cr^{-(2-\s)n}.
\end{align*}
\end{enumerate}
\end{lemma}

\begin{proof}
To prove 1 we apply the Corollary \ref{cor:key_corollary} to $(u - P_{p,h})$ in $C_{2^{-k}r,\D t}(x,t)$. By Theorem \ref{thm:equation_for_the_diference} and Lemma \ref{lem:MP_greater_than_zero} we know that $u$ satisfies in viscosity in $C_{1,1}(x,0) \supseteq C_{2^{-k}r,\D t}(x,t)$,
\begin{align*}
u_t - \cM_{\cL_0}^-u &\geq -1 + \cM_{\cK_0}^-P_{p,h} - \b|p| \geq -C.
\end{align*}
This proves (1).

To prove (2) we notice first that the previous estimate for $(u - P_{p,h})$ also holds for the parabolic convex function $(\G - P_{p,h}) \leq (u - P_{p,h})$. Then by Lemma \ref{tapaarriba} we get the upper bound $\G - P_{p,h} \leq Cr^{-(2-\s)}r_i^2$ in $C_{r_i/2,\D t}(x,t+\D t/2)$. Meanwhile, the lower bound holds by the definition of the convex envelope and Corollary \ref{cor:h}.

As a consequence of the bounds given by (2) and the geometry of convex functions we get that $\diam( D\G(B_{r_i/4}(x) \times \{t+\D t\})) \leq Cr^{-(2-\s)}r_i$. Then, by Lemma \ref{lem:properties_h},
\begin{align*}
\Phi(C_{r_i/4,\D t}(x,t+\D t)) \ss & \ Cylinder,\\
:= & \ \{(p',h'):p' \in  D\G(B_{r_i/4}(x) \times \{t+\D t\}),\\
& \ h' \in [h(p',t),h(p',t) + C\D t]\}.
\end{align*}
For which it is easy to verify that 
\[
|Cylinder| \leq Cr^{-(2-\s)n}r_i^n\D t.
\]
This concludes (3) and the proof of the lemma.
\end{proof}

\begin{theorem}[Covering Lemma for the Contact Set]\label{thm:abp}
Let $k \sim (2-\s)^{-1}$ as in Lemma \ref{lem:key_lemma}, $r \in (0,1]$, $\D t \in (0,(2^{-k}r)^2]$, $t \in (-1+3\D t/2,-\D t]$ and $J = (t-\D t/2,t]$. There exists a finite family of disjoint open boxes $\{K_j\}$ such that:
\begin{enumerate}
 \item  $K_j := Q_j\times J$ with $Q_j \ss \R^n$ an open cube with $\diam(Q_j) < r$.
 \item $\overline{K_j} \cap \Sigma \neq \emptyset$ and $\Sigma \cap (\R^n \times J) \ss \bigcup_j \overline{K_j}$.
 \item \label{c} \textbf{Control of $u$ detaching from $\G$:}
 \begin{align*}
  \frac{|\{u \leq \G + C\}\cap \tilde K_j|}{|\tilde K_j|} \geq \m,
 \end{align*}
 where,
 \begin{align*}
 \tilde K_j &:= \tilde Q_j\times \tilde J,\\
 \tilde Q_j &:= 16\sqrt n Q_j,\\
 \tilde J &:= (t-3\D t/2,t].
 \end{align*}
 \item \label{b} \textbf{Control of the jacobian measure of $\Phi$:}
 \begin{align*}
  \frac{|\Phi(\bar K_j)|}{|K_j|} \leq Cr^{-(2-\s)n}.
 \end{align*}
\end{enumerate}
For some universal constants $C>0$ and $\m \in (0,1)$ independent of $\s \in [1,2)$.
\end{theorem}

\begin{proof}
Consider a covering of $B_1$ contained in $B_2$ by congruent cubes $\{Q\}$ with $\diam(Q) = r/4$. Discard every rectangle $K = Q \times J$ such that $\overline{K} \cap \Sigma = \emptyset$. Whenever $Q\times J$ does not satisfy (\ref{b}) or (\ref{c}), we split $Q$ into $2^n$ congruent cubes $\{Q'\}$ and consider now the rectangles given by $\{K' = Q'\times J\}$. We need to prove that eventually all rectangles produced by this algorithm satisfy (\ref{c}) and (\ref{b}). In fact we will show that it will finish before $k \sim (2-\s)^{-1}$ iterations.

Let $\overline{Q_0} \times J \supseteq \overline{Q_2} \times J \supseteq \ldots \supseteq \overline{Q_{k-1}} \times J \ni (x_0,t_0)$ such that $(x,t) \in \Sigma$ and $\diam(Q_i) = (r/4)2^{-i}$. Let also $p\in\p\G(x,t)$ and $h = h(p,t)$. From Lemma \ref{lem:covering_piece} there exists some non negative integer $i \leq k$, such that for $r_i = 2^{-i}r$,
\begin{align}
\label{eq:abp1} &\frac{|\{u > P_{p,h} + Cr^{-(2-\s)}r_i^2\} \cap R_{r_i,\D t}(x,t)|}{|R_{r_i,\D t}(x,t)|} \leq \e_0,
\end{align}
and
\begin{align}
\label{eq:abp2} &\frac{|\Phi(C_{r_i/4,\D t}(x,t+\D t/2))|}{|C_{r_i/4,\D t}(x,t+\D t/2)|} \leq Cr^{-(2-\s)n}.
\end{align}
Now,
\begin{align*}
\diam(Q_i) = r_i/4 \Rightarrow \begin{cases}
K_i \ss C_{r_i/4,\D t}(x,t+\D t/2),\\
R_{r_i,\D t}(x,t) \ss \tilde K_i.
\end{cases}
\end{align*}
Then (\ref{c}) and (\ref{b}) follow from \eqref{eq:abp1}, \eqref{eq:abp2}, the previous inclusions and the fact that we can replace $P_{p,h}$ by $\G \geq P_{p,h}$ in \eqref{eq:abp1}.
\end{proof}

\section{Point Estimate}\label{section:Point Estimate}

In this section we prove Point Estimate as the one in \cite{Davila12-p} for operators with non-symmetric kernels. In the classical case one can use that, at sufficiently small scales, the drift becomes so small that it can be absorbed by the estimates that can be proved for pure second order equations. In our case this is no longer necessarily true and provide us with new challenges. 

\begin{theorem}\label{thm:PE}
Let $\s \in [1,2)$. Suppose $u \geq 0$ satisfies 
\begin{align*}
u_t - \cM^-_{\cL_0}u &\geq -1 \text{ in $C_{2,2}(0,1)$},\\
\inf_{C_{1,1}(0,1)}u &\leq 1.
\end{align*}
Then, for every $s \geq 0$, 
\begin{align*}
|\{u > s\} \cap C_{1,1}| \leq Cs^{-\e},
\end{align*}
for some constants $\e$, $C$ depending only on $n, \l, \L$ and $\b$.
\end{theorem}

Corollary \ref{cor:weak_point_estimate} already tells us that the result holds if we are willing to allow constants that degenerate as $\s$ goes to two. Therefore we can restrict the analysis for values of $\s$ close to two where we can construct special functions based on how the second order operator evaluates on them.

In the case we have a right hand side $-f(t) \in L^1((-1,0])$ we can apply the previous Theorem to $\tilde u = u + \int_{-1}^t f^+(s)ds$ and recover the estimate in $L^1_t$.

\begin{corollary}
Let $\s \in [1,2)$. Suppose $u \geq 0$ satisfies 
\begin{align*}
u_t - \cM^-_{\cL_0}u &\geq -f(t) \text{ in $C_{2r,2r^\s}(0,r^\s)$}.
\end{align*}
Then, for every $s \geq 0$, 
\begin{align*}
\frac{|\{u > s\} \cap C_{r,r^\s}|}{|C_{r,r^\s}|} \leq C\1\inf_{C_{r,r^\s}(0,r^\s)}u + \int_{-r^\s}^{r^\s}f^+(t)dt\2^{\e}s^{-\e},
\end{align*}
for some constants $\e$, $C$ depending only on $n, \l, \L$ and $\b$.
\end{corollary}

The proof of the Point Estimate, Theorem \ref{thm:PE}, follows some of the steps as in Section 5 from \cite{Davila12-p}. In general lines, the estimate is proven inductively over level sets of the form $M^k$ with $M>1$ universal and $k=1,2,\ldots$. The main difficulty related with the drift is contained in the first step and taken care by a barrier described by Lemma \ref{lemma:Barrier2}.

\textbf{Step 1:} We get the estimate for $k=1$ by using two auxiliary functions before applying the previous ABP type of estimate. The first lemma will be used to concentrate the positive part of the right-hand side of the equation inside of $C_{1/8,1}$. However, if we consider the gradient term as a right-hand side this still spreads its positive part all over the domain. To manage this we use a second auxiliary function which controls the growth of the solution from below starting from $B_{1/8}$. After a vertical translation we get that the negative part of the solution, and its contact set, get inside $B_{1/4}$.

By a standard rescaling and covering procedure we can rewrite the main goal of this step in the following way. Recall that $K_{r,\t}(x,t) = Q_r(x) \times (t-\t,t]$ where $Q_r(x)$ is the spatial cube of side $r$ centered at $x \in \R^n$. Whenever $(x,t)$ is omitted it is assumed that the box is centered at the origin.

\begin{lemma}[Base configuration] \label{lem:base}
Let $\s$ as in Lemma \ref{lem:barrier_pt_est} and suppose $u \geq 0$ satisfies,
\begin{alignat*}{2}
u_t-\cM^-_{\cL_0}u &\geq -1 &&\text{ in $C_{2\sqrt n,37}(0,36)$},\\
\inf_{K_{3,36}(0,36)}u &\leq 1.
\end{alignat*}
Then
\begin{align*}
\left|\3u \geq M_1\4\cap K_{1,1}\right| \leq \m_1|K_{1,1}|,
\end{align*}
for some universal constants $\m_1\in(0,1)$, $M_1>1$ independent of $\s$.
\end{lemma}

\begin{lemma}[Special Function]\label{lem:barrier_pt_est}
There exists a smooth function $\varphi$ such that,
\begin{alignat*}{2}
\varphi_{t}-\cM^-_{\cL_0}\varphi &\leq -1 &&\text{ in $C_{2\sqrt n,37}(0,36) \sm C_{1/8,1}$},\\
\varphi &\leq 0 &&\text{ in $\p_pC_{2\sqrt n,37}(0,36)$},\\
\varphi &\geq 2 &&\text{ in $K_{3,36}(0,36)$},
\end{alignat*}
provided that $\s$ is sufficiently close to two.
\end{lemma}

\begin{proof}
The Lemma says that $\varphi$ should behave similarly to the fundamental solution in the sense that it allows $\varphi$ to grow a fix amount at $K_{3,36}(0,36)$ by concentrating the positive part of its right hand side in the intermediate region $C_{1/8,1}$. This is our initial Ansatz,
\begin{align*}
\varphi_1(y,s) &:= (s+1)^{-\a^3}\Phi(y(s+1)^{-1/2}),\\
\Phi(z) &:= \exp(-(\a/2)|z|^2).
\end{align*}

\textbf{Step 1:} $(\varphi_1)_t + \b|D\varphi_1| - \inf_{A \in [\l',\L']}\trace(AD^2\varphi_1) \leq -\a^{3/2}(s+1)^{-(\a^3+1)}\Phi$ in $C_{2\sqrt n,37}(0,36) \sm C_{1/8,1}$ provided that $\a$ is sufficiently large. Note now,
\begin{align*}
(\varphi_1)_t &= \a (s+1)^{-(\a^3+1)}(-\a^2 + (1/2)|z|^2)\Phi,\\
|D\varphi_1| &= \a (s+1)^{-(\a^3+1/2)}|z|\Phi,\\
D^2\varphi_1 &= \a (s+1)^{-(\a^3+1)}(\a z\otimes z - Id)\Phi.
\end{align*}
For $(y,s) \in C_{1,1} \sm C_{1/8,1}$ and $\a$ sufficiently large,
\begin{align*}
\inf_{A \in [\l',\L']}\trace(AD^2\varphi_1) &= \a(s+1)^{-(\a^3+1)}(\l'(\a|z|^2-1) - \L'(n-1))\Phi,\\
&\geq \frac{\l'\a^2}{128}(s+1)^{-(\a^3+1)}\Phi.
\end{align*}
We note that this term controls all the other terms in the operator in $C_{1,1} \sm C_{1/8,1}$.

In $C_{2\sqrt n,36}(0,36)$ we can not use anymore that $y$ is away from zero. We use instead the ``good'' term coming from $(\varphi_1)_t$, the fact that $(s+1) \sim 1$ and $|z| \leq 2 \sqrt n$. It is not difficult to see that the leading order will be $-\a^3 (s+1)^{-(\a^3+1)}$ for $\a$ is large.

\textbf{Step 2:} Let $\psi$ be a smooth function such that $\psi = 0$ in $\p_pC_{2\sqrt n,37}(0,36) \cup (\R^n \times [-1,-1/2])$ and $\psi = 1$ in $K_{3,36}(0,36)$. We consider now $\varphi_2 := \psi\varphi_1$, by the previous step,
\begin{align*}
(\varphi_2)_t + \b|D\varphi_2| - \inf_{A \in [\l',\L']}\trace(AD^2\varphi_2) &\leq \psi(\varphi_t + \b|D\varphi| - \inf_{A \in [\l',\L']}\trace(AD^2\varphi))\\
&{} + \varphi(\psi_t + \b|D\psi| - \inf_{A \in [\l',\L']}\trace(AD^2\psi))\\
&{} + 2\inf_{A \in [\l',\L']}\trace(A D\psi \otimes D\varphi),\\
&\leq C(-\a^{3/2}(s+1)^{-1} + 1 + \a (s+1)^{-1})\\
&(s+1)^{-\a^3}\Phi.
\end{align*}
Taking $\a$ sufficiently large guarantees that the right hand side above is non positive in $C_{2\sqrt n,37}(0,36) \sm C_{1/8,1}$.

\textbf{Step 3:} Let
\begin{align*}
\varphi(y,s) := C\1\varphi_2(y,s) - (1/100)\1\inf_{K_{3,36}(0,36)} \varphi_2\2(s+1)\2.
\end{align*}
We choose $C$ sufficiently large so that
\begin{alignat*}{2}
\varphi_t + \b|D\varphi| - \inf_{A \in [\l',\L']}\trace(AD^2\varphi) &\leq -2 &&\text{ in $C_{2\sqrt n,37}(0,36) \sm C_{1/8,1}$},\\
\varphi &\geq 2 &&\text{ in $K_{3,36}(0,36)$}.
\end{alignat*}
Notice now that in the closure of $C_{2\sqrt n,37}(0,36) \sm C_{1/8,1}$, $\varphi_t-\cM^-_{\cL_0}\varphi$ goes uniformly below $-1$, as $\s$ goes to two. This is how we choose $\s$ close to 2 in order to conclude the Lemma.
\end{proof}

%%%

Let $u$ satisfy the hypotheses of Lemma \ref{lem:base}. By considering,
\begin{align*}
v = u-\varphi-\inf (u-\varphi) - 1
\end{align*}
we obtain that $v$ satisfies an equation with right hand side concentrated in $C_{1/8,1}$. However, this property gets lost if we move the gradient term to the right hand side. This is why we consider the following lemma in order to localize the region where to apply the ABP estimate around $C_{1/8,1}$.

By proving an estimate for the distribution of $v$ given above, we get an estimate for the distribution of $u$ given that $\varphi$ is universal. The following will be the hypothesis of the next lemma, for some $C_0>0$ universal coming from the construction of $\varphi$, we have:
\begin{align*}
v_t - \cM^-_{\cL_0}v &\geq -C_0\chi_{B_{1/8}} \text{ in $C_{2\sqrt n,1}$},\\
v &\geq 0 \text{ in $\p_p C_{2\sqrt n,1}$},\\
\sup v^- &= 1.
\end{align*}
Notice that for $t>0$, $v_t - \cM^-_{\cL_0}v  \geq 0$, which implies that the infimum of $u$ is attained for some $t \in (-1,0]$.

\begin{lemma}\label{lemma:Barrier2}
Given $\s$ sufficiently close to two and $v$ satisfying the previous hypotheses, there is some $\a>2$ sufficiently large and independent of $\s$ such that,
\begin{align*}
\psi(x) := \1\1\frac{(|x|-1/8)^+}{2\sqrt n}\2^\a-1\2\chi_{B_{2\sqrt n}}(x) \leq v(x,t).
\end{align*}
\end{lemma}

\begin{proof}
It suffices to show that $\cM^-_{\cL_0} \psi \geq 0$ in $B_{2\sqrt n} \sm \bar B_{1/8}$, for $\a$ sufficiently large to conclude that $\psi \leq v$.

By a standard computation we have for $x \in B_{2\sqrt n}\sm \bar B_{1/8}$, $\a > 2$ and $|x| = r$,
\begin{align*}
\inf_{A \in [\l',\L']} \trace (AD^2\psi) - \b|D\psi| &\geq \frac{\a(r - 1/8)^{\a-2}}{(2\sqrt n)^\a}\1 \l'(\a-1) - \b\2.
\end{align*}
So for $\a-1 > \b/\l'$ and $\s$ sufficiently close to two, we get that $\cM^-_{\cL_0} \psi \geq 0$ in $B_{2\sqrt n} \sm \bar B_{1/8}$. Notice that $\psi = 0$ outside of $B_{2\sqrt n}$ and strictly negative in $B_{2\sqrt n}$, which is important to keep the sign of the operator as $r$ approaches $1/8$.
\end{proof}

\begin{proof}[Proof of Lemma \ref{lem:base}]
Let
\begin{align*}
v = u - \varphi - \inf (u-\varphi) - (16\sqrt n)^{-\a}.
\end{align*}
By the previous lemmas we know that for some universal $C_0>0$,
\begin{align*}
v_t - \cM^-_{\cL_0}v &\geq -C_0\chi_{B_{1/8}} \text{ in $C_{2\sqrt n,1}$},\\
v &\geq 0 \text{ in $\p_p C_{1/4,1}$},\\
\sup v^- &= (16\sqrt n)^{-\a}.
\end{align*}
Now we apply a rescaled version of Theorem \ref{thm:abp} in each time interval
\begin{align*}
J^{(l)} &= [t^{(l)}-\D t/2, t^{(l)}],\\
t^{(l)} &= -l\D t/2 \text{ such that $l \in \N$ and $t^{(l)} \in (-1+3\D t/2, -\D t]$}
\end{align*}
for some universal $r$ sufficiently small such that each cube $\tilde Q^{(l)}_j \ss Q_1$.

The key idea in \cite{Tso85} is to notice that $\Phi(B_{2\sqrt n}\times(-1,0])$ contains a fixed ball of size comparable to $\sup v^-$, which is universal in our case, by bringing (spatial) planes from the past until they hit the graph of $v$ at some point in the contact set. We get in this way that $|\Phi(B_{2\sqrt n}\times(-1,0])|$ is greater than some universal constant. Then we can proceed as in \cite{Davila12-p} to recover the following estimate for the distribution of $v$ from the ABP type of estimate given in Theorem \ref{thm:abp},
\begin{align*}
|\{v \geq M_1\} \cap K_{1,1}| \leq \m_1|K_{1,1}|.
\end{align*}
\end{proof}

\textbf{Step 2:} Lemma \ref{lem:base} will be applied at smaller scales but we need some flexibility on the shape of the domain (given by a parameter $\t \in [1,4]$). The reason will become clear on the next step when we introduce a particular dyadic decomposition which takes into account the scaling of the equation. Keep in mind that the operators $\cM^\pm_{\cL_0}$ are invariant by scaling. %We have in mind the following corollary which also includes iterations of Lemma \ref{lem:base}.

\begin{corollary}[Iteration]\label{cor:iteration}
Let $\t \in [1,4]$, $m\in\N$, $d_i = \sum_{j=1}^i 3^{\s j}$, and suppose $u \geq 0$ satisfies,
\begin{align*}
u_t-\cM^-_{\cL_0}u &\geq -1 \text{ in $C_{3^{\s(m-1)}2\sqrt n,d_m+\t}(0,d_m\t)$},\\
\inf_{\cup_{i=1}^m K_{3^i,3^{\s i}\t}(0,d_i\t)}u &\leq 1.
\end{align*}
Then,
\begin{align*}
\left|\3u \geq M_2^m\4\cap K_{1,\t}\right| \leq \m_2|K_{1,\t}|,
\end{align*}
for some universal constants $\m_2\in(0,1)$, $M_2>1$ independent of $\s$.
\end{corollary}

\textbf{Step 3:} We consider now the following dyadic decomposition of $K_{1,1}$ which will allow us to iterate the estimate in measure at smaller scales and for higher level sets.

We start with the box $K_{1,1}$. In each step we consider one of the boxes $K_{r,r^\s\t}(x,t)$ we have already produced and divide $Q_r(x)$ in $2^n$ congruent cubes in space. With respect to the time interval we do the following:
\begin{enumerate}
\item If $1\leq\t<2$ then we subdivide $[t-r^\s\t,t]$ in 2 congruent intervals.
\item If $2\leq\t\leq 4$ then we subdivide $[t-r^\s\t,t]$ in 4 congruent intervals.
\end{enumerate}
Finally we take the cartesian product to form the new generation of dyadic boxes from $K_{r,r^\s\t}(x,t)$.

This procedure satisfies that if $\t\in[1,4]$, then the boxes that $K_{r,r^\s\t}(x,t)$ generates have side length $r/2$ (in space) and $(r/2)^\s\t'$ (in time) for some $\t'\in[1,4]$.

Given two dyadic boxes $K$ and $\tilde K$ we say that $\tilde K$ is the predecessor of $K$ if $K$ is one of the boxes obtained from the decomposition of $\tilde K$.

Due to the configuration of Lemma \ref{lem:base} we need to introduce a new tool which handles a shift in time. Given a box $K = Q\times(t-\t,t]$ and a natural number $m \geq 1$ we define the $m$-stack $K^m := Q\times(t,t+m\t]$.

As a consequence of the Lebesgue Differentiation Theorem we get the following lemma.

\begin{lemma}\label{lem:cz}
Given $A$ with density at most $\m\in(0,1)$ in $K_{1,1}$,
\begin{align*}
\frac{|A \cap K_{1,1}|}{|K_{1,1}|} \leq \m.
\end{align*}
There exists a set of disjoint dyadic boxes $\{K_j\}$ such that:
\begin{enumerate}
\item \textbf{Covers in measure}: $\displaystyle \left|\bigcup_j K_j \sm A\right| = 0$,
\item \textbf{The density of $A$ in each piece is greater than $\mu$}: 
\begin{align*}
\frac{|A \cap K_j|}{|K_j|} > \m,
\end{align*}
\item \textbf{The density of $A$ in the union of $m$-stacks of predecessors is at most $\mu_m := (m+1)\mu/m \sim \m$ ($m$ large)}: 
\begin{align*}
\frac{\left|A \cap \bigcup_j (\tilde K_j)^m\right|}{\left|\bigcup_j (\tilde K_j)^m\right|} \leq \mu_m.
\end{align*}
\end{enumerate}
\end{lemma}

In the following lemma we use the previous covering and prove some diminishment of the distribution. However, it has the disadvantage that can not be iterated. The reason being that the smaller level set is measured in $K_{1,d_m+1}(0,d_m)$ instead of $K_{1,1}$. Still it is an important step for the proof of Theorem \ref{thm:PE}.

\begin{lemma}[Raw Diminish of Distribution]\label{lem:stack}
Let $\s,m$ (large), $d_i$ as in Corollary \ref{cor:iteration} and $u \geq 0$ such that,
\begin{align*}
u_t-\cM^-_{\cL_0}u &\geq -1 \text{ in $C_{3^{\s(m-1)}2\sqrt n,d_m+1}(0,d_m)$},\\
\inf_{\cup_{i=1}^m K_{3^i,3^{\s i}}(0,d_i)}u &\leq 1.
\end{align*}
Then, for any $i \in \N$ and for $\m_m (= \frac{m+1}{m}\m_2) \in (0,1)$ and $M_m (= M_2^m) > 1$, we have that,
\begin{align*}
\left|\3u \geq M_m^{i+1}\4 \cap K_{1,1}\right| \leq \m_m\left|\3u \geq M_m^{i}\4 \cap K_{1,d_m+1}(0,d_m)\right|.
\end{align*}
\end{lemma}

\textbf{Step 4:} Theorem \ref{thm:PE} follows from the next diminishment of the distribution lemma. It improves Lemma \ref{lem:stack} due to the fact that the level sets are measured in $K_{1,1}$. The proof goes by contradiction using Lemma \ref{lem:stack} and the dyadic decomposition from the previous step. The details can be found in \cite{Davila12}.

\begin{lemma}[Diminish of Distribution]\label{lem:discrete_point_estimate}
Let $\s,m$ (large), $d_i$ as in Corollary \ref{cor:iteration} and $u \geq 0$ such that,
\begin{align*}
u_t-\cM^-_{\cL_0}u &\geq -1 \text{ in } C_{3^{\s(m-1)}2\sqrt n,d_m+1}(0,d_m),\\
\inf_{K_{1,1}(0,17)}u &\leq 1.
\end{align*}
Then, for any $k \in \N$,
\begin{align*}
 |\{u \geq M^{k+1}\}\cap K_{1,1}| \leq \m|\{u \geq M^k\}\cap K_{1,1}|,
\end{align*}
for $\m\in(0,1)$ and $M>1$ universal and independent of $\s$.
\end{lemma}

\section{Oscillation Lemma and Harnack Inequality}\label{section:Oscillation Lemma and Harnack Inequality}

Both results announced in the title of this section were left open in \cite{Davila12-p}. By an oscillation lemma we mean a pointwise estimate for non negative sub solutions in terms of an integral norm, it is preliminary to the well known Harnack type of estimate. At the end of this section we will state some of the H\"older estimates that result as corollaries.

Here is some notation needed for the rest of the section.
\begin{align*}
\varphi(y,s) &= \varphi(|y|,s) := ((1+s)^{1/\s}-|y|)^{-(n+\s)},\\
P_r(x,t) &:= \{(y,s) \in \R^{n+1}: |y-x|^\s - r^\s \leq s-t \leq 0\},\\
P_r &:= P_r(0,0).
\end{align*}

\begin{lemma}[Oscillation Lemma]\label{lem:oscillation}
Let $u \geq 0$ satisfies,
\begin{alignat*}{2}
u_t - \cM^+_{\cL_0} u &\leq f(t) &&\text{ in $P_1$},\\
\|f^+\|_{L^1(-1,0]} &\leq 1,\\
\|u\|_{L^1((-1,0]\mapsto L^1(\w_\s))} &\leq 1.
\end{alignat*}
Then
\begin{align*}
u \leq C\varphi \text{ in $P_1$},
\end{align*}
for some universal $C>0$.
\end{lemma}

The previous lemma implies the following corollary.

\begin{corollary}\label{cor:oscillation1}
Let $\cL\ss \cL_0$, $I:\W\times(t_1,t_2]\times\R^\cL\to\R$ uniformly elliptic and such that $I0=0$. Let $u$ satisfies,
\begin{align*}
u_t - I u &\leq f \text{ in $\W\times(t_1,t_2]$}.
\end{align*}
Then for every $\W'\times(t_1',t_2] \cc \W\times(t_1,t_2]$,
\begin{align*} 
\sup_{\W'\times(t_1',t_2]} u^+ \leq C\1\|u^+\|_{L^1((t_1,t_2] \mapsto L^1(\w_\s))} + \|f^+\|_{L^1((t_1,t_2]\mapsto L^\8(\W))}\2,
\end{align*}
for some universal $C>0$ depending on the domains.
\end{corollary}

\begin{proof}[Proof of Lemma \ref{lem:oscillation}]
Let $M := \inf\{ M' \in \R^+: u \leq M'\varphi \text{ in } P_1\}$. It suffices to show that $M$ is universally bounded.

Let $(x,t) \in P_1$ such that $u_0 := u(x,t) = M\varphi(x,t) = Md^{-(n+\s)}$ where $d :=  ((1 + t)^{1/\s} - |x|) > 0$. Since $u \in L^1(P_1)$, we have that, for any $\theta\in(0,1)$,
\begin{align*}
 \left|\3 u \geq \frac{u_0}{2} \4 \cap P_{\theta d/8}(x,t)\right| &\leq Cu_0^{-1} \leq C\theta^{-(n+\s)}M^{-1}\left|P_{\theta d/8}(x,t)\right|.
\end{align*}
As we want to get a contradiction if $M$ is arbitrarily large, we will show that
\begin{align}
 \left|\3 u < \frac{u_0}{2} \4\cap P_{\theta d/8}(x,t)\right| \leq \frac{1}{2}\left|P_{\theta d/8}\right|,
\end{align}
for some $\theta>0$ to be fixed, independent of $M$.

Let $r := d/2$. On $P_{\theta r}(x,t)$, $u$ is bounded from above by $M\varphi(|x| + \theta r,t) = u_0(1-\theta/2)^{-(n+\s)}$. Indeed, by the geometry of the paraboloids that we have constructed, it is not difficult to check that the closest point to $P_1^c$ in $P_{\theta r}(x,t)$ gets realized at time $t$ and at some point in $\p B_{\theta r}(x)$. %See figure \ref{fig:closest_pt_paraboloid}.
%\begin{figure}
 %\begin{center}
  %\includegraphics[height=5cm]{./img/closest_pt_paraboloid.eps}
 %\end{center}
 %\caption{Closest point to to $P_1^c$ in $P_{\theta r}(x,t)$.}
 %\label{fig:closest_pt_paraboloid}
%\end{figure}

Consider $w = (u_0(1-\theta/2)^{-(n+\s)} - u)^+$, which is just $(u_0(1-\theta/2)^{-(n+\s)} - u)$ in $P_{\theta r}(x,t)$. In the smaller domain $P_{\theta r/2}(x,t)$, $w$ satisfies an equation with a right hand side that includes a contribution coming from the truncation,
\begin{align*}
 w_t - \cM^-_{\cL_0} w &\geq -f - \cM^+_{\cL_0}(w+u) \text{ in $P_{\theta r/2}(x,t)$}
\end{align*}
We want to show that for $(y,s) \in P_{\theta r/2}(x,t)$, we have
\begin{align}\label{eq:rhs_osc_lemma}
\int_{t-(\theta r/2)^\s}^t F(s) ds \leq C(\theta r)^{-(n+\s)},
\end{align}
where
\begin{align*}
F(s) := \sup_{(y,s) \in P_{\theta r/2}(x,t)}f(s) + \cM^+_{\cL_0}(w+u)(y,s).
\end{align*}
Note that for $(y,s) \in P_{\theta r/2}(x,t)$,
\begin{align*}
\d(w+u)(y;z) &= (w+u)(y+z) - u_0(1-\theta/2)^{-(n+\s)} - 0,\\
&= \1u(y+z)-u_0(1-\theta/2)^{-(n+\s)}\2^+,\\
&\leq u(y+z)\chi_{B^c_{\theta r/2}}(z),
\end{align*}
and so we get
\begin{align*}
\cM^+_{\cL_0}(w+u)(y,s) &\leq C\int_{B^c_{\theta r/2}}\frac{u(z+y,s)}{|z|^{n+\s}}dz,\\
&\leq C(\theta r)^{-(n+\s)}\|u(s)\|_{L^1(\w_\s)}.
\end{align*}
This proves \eqref{eq:rhs_osc_lemma} assuming that $\theta r$ is sufficiently small to dominate over $\|f^+\|_{L^1} \leq 1$.

Now we apply the Point Estimate Theorem \ref{thm:PE} to $w$ with respect to the domains $P_{\theta r/2}(x,t)$ and $P_{\theta r/4}(x,t)$. To justify it you may use a standard covering argument.
\begin{align*}
&\frac{\left|\3u<u_0/2\4 \cap P_{\theta r/4}(x,t)\right|}{|P_{\theta r/4}(x,t)|},\\
= \ &\frac{\left|\3w > u_0((1-\theta/2)^{-(n+\s)}-1/2)\4 \cap P_{\theta r/4}(x,t)\right|}{|P_{\theta r/4}(x,t)|},\\
\leq \ &C(w(x,t) + (\theta r)^{-(n+\s)}(\theta r)^\s)^\e(u_0((1-\theta/2)^{-(n+\s)} - 1/2))^{-\e}.
\end{align*}
Now, by taking $\theta$ sufficiently small such that $(1-\theta/2)^{-(n+\s)} \geq 3/4$ and dropping the factor $(\theta r)^\s \leq 1$, we get
\begin{align*}
\frac{\left|\3u<u_0/2\4 \cap P_{\theta r/4}(x,t)\right|}{|P_{\theta r/4}(x,t)|}  &\leq C(w(x,t) + (\theta r)^{-(n+\s)})^\e u_0^{-\e},\\
&\leq C(((1-\theta/2)^{-(n+\s)}-1)+M^{-1}\theta^{-(n+\s)})^\e,\\
&\leq C[((1-\theta/2)^{-(n+\s)}-1)^\e+(M^{-1}\theta^{-(n+\s)})^\e].
\end{align*}
Finally, fix $\theta$ even smaller such that $C((1-\theta/2)^{-(n+\s)}-1)^\e \leq 1/4$ which implies that for $M$ sufficiently large the right hand side above becomes smaller than $1/2$, independently from the now fixed $\theta$.
\end{proof}

The same technique used for Lemma \ref{lem:oscillation} allows us to handle the Harnack Inequality for this class of parabolic non-local operators.

For the following result we fix $\e>0$, the one given by the Point estimate Theorem \ref{thm:PE}, and define $\varphi_\e$ as
\begin{align*}
\varphi_\e(|y|,s)= ((1+s)^{1/\s}-|y|)^{-(n+\s)/\e}.
\end{align*}

\begin{theorem}[Harnack Inequality]
Let $u\geq0$ satisfy
\begin{alignat*}{2}
 u(0,1) &\leq 1,\\
 u_t - \cM^+_{\cL_0}u &\leq f(t), &&\text{ in $C_{1,2}(0,1)$},\\
 u_t - \cM^-_{\cL_0}u &\geq -f(t), &&\text{ in $C_{1,2}(0,1)$},\\
 \|f\|_{L^1(-1,1]} &\leq 1.
\end{alignat*}
Then
\begin{align*}
 u \leq C\varphi_\e \text{ in $P_1$},
\end{align*}
for some universal constant $C$.
\end{theorem}

\begin{corollary}\label{cor:harnack}
Let $\cL\ss \cL_0$, $I:C_{2,4}\times\R^\cL\to\R$ uniformly elliptic and such that $I0=0$. Let $u\geq0$ satisfy,
\begin{align*}
u_t - I u &= f(t) \text{ in $C_{2,4}$}.
\end{align*}
Then,
\begin{align*} 
\sup_{C_{1,1}(0,-2)} u \leq C\1\inf_{C_{1,1}}u + \|f^+\|_{L^1((-4,0]\mapsto L^\8(B_2))}\2,
\end{align*}
for some universal $C>0$.
\end{corollary}

%\begin{theorem}[Harnack Inequality]
%Let $u\geq0$ satisfies
%\begin{alignat*}{2}
% u(0,0) &\leq 1,\\
% u_t - \cM^+_{\cL_0}u &\leq f(t), &&\text{ in $C_{1,1}$},\\
% u_t - \cM^-_{\cL_0}u &\geq -f(t), &&\text{ in $C_{1,1}$},\\
% \|f\|_{L^1(-1,1]} &\leq 1.
%\end{alignat*}
%Then
%\begin{align*}
% u \leq C \text{ in $C_{1/4,1/4}$},
%\end{align*}
%for some universal constant $C$.
%\end{theorem}

\begin{proof}

From Corollary \ref{cor:weak_point_estimate} we use the super solution inequality to get,
\begin{align}\label{eq:harnack}
(2-\s)\int_{-1}^0 \|u(s)\|_{L^1(\w_\s)}ds \leq C.
\end{align}
This bound will be important when estimating the (non-local) error that comes out when we truncate the solution. 

Let $M := \inf\{ M' \in \R^+: u \leq M'\varphi_\e \text{ in } P_1\}$. It suffices to show that $M$ is universally bounded.

Let $(x,t) \in P_1$ such that $u_0 := u(x,t) = M\varphi_\e(x,t) = Md^{-(n+\s)/\e}$ where $d :=  ((1 + t)^{1/\s} - |x|) > 0$. Thanks to Theorem \ref{thm:PE}, we have that, for any $\theta\in(0,1)$,
\begin{align*}
 \left|\3 u \geq \frac{u_0}{2} \4 \cap P_{\theta d/8}(x,t)\right| &\leq Cu_0^{-\e} \leq C\theta^{-(n+\s)}M^{-\e}\left|P_{\theta d/8}(x,t)\right|.
\end{align*}
As we want to get a contradiction if $M$ is arbitrarily large, we will show that
\begin{align}
 \left|\3 u < \frac{u_0}{2} \4\cap P_{\theta d/8}(x,t)\right| \leq \frac{1}{2}\left|P_{\theta d/8}\right|,
\end{align}
for some $\theta>0$ to be fixed, independently of $M$.

As in the proof of the Oscillation Lemma, let $r := d/2$ and note that over $P_{\theta r}(x,t)$, $u$ is bounded from above by $u_0(1-\theta/2)^{-(n+\s)/\e}$. %Indeed, by the geometry of the paraboloids that we have constructed, it is not difficult to check that the closest point to $P_1^c$ in $P_{\theta r}(x,t)$ gets realized at time $t$ and at some point in $\p B_{\theta r}(x)$. See figure \ref{fig:closest_pt_paraboloid}.
%\begin{figure}
 %\begin{center}
  %\includegraphics[height=5cm]{./img/closest_pt_paraboloid.eps}
 %\end{center}
 %\caption{Closest point to to $P_1^c$ in $P_{\theta r}(x,t)$.}
 %\label{fig:closest_pt_paraboloid}
%\end{figure}

Consider $w = (u_0(1-\theta/2)^{-(n+\s)/\e} - u)^+$, which is equal to $(u_0(1-\theta/2)^{-(n+\s)/\e} - u)$ in $P_{\theta r}(x,t)$. In the smaller domain $P_{\theta r/2}(x,t)$, $w$ satisfies an equation with a right hand side that includes a contribution coming from the truncation,
\begin{align*}
 w_t - \cM^-_{\cL_0} w &\geq -f - \cM^+_{\cL_0}(w+u) \text{ in $P_{\theta r/2}(x,t)$}
\end{align*}
We want to show that for $(y,s) \in P_{\theta r/2}(x,t)$, we have
\begin{align}\label{eq:rhs_harnack}
\int_{t-(\theta r/2)^\s}^t F(s) ds \leq C(\theta r)^{-(n+\s)},
\end{align}
where
\begin{align*}
F(s) := \sup_{(y,s) \in P_{\theta r/2}(x,t)}\1 f(s) + \cM^+_{\cL_0}(w+u)(y,s)\2.
\end{align*}
Because of the bound \eqref{eq:harnack} we get that for $(y,s) \in P_{\theta r/2}(x,t)$
\begin{align*}
\int_{t-(\theta r/2)^\s}^t \cM^+_{\cL_0}(w+u)(y,s) ds &\leq C(\theta r)^{-(n+\s)},
\end{align*}
Which implies the bound \eqref{eq:rhs_harnack}, assuming that $\theta r$ is sufficiently small to dominate over $\|f^+\|_{L^1} \leq 1$.

Now we apply the Point Estimate Theorem \ref{thm:PE} to $w$ with respect to the domains $P_{\theta r/2}(x,t)$ and $P_{\theta r/4}(x,t)$.
\begin{align*}
&\frac{\left|\3u<u_0/2\4 \cap P_{\theta r/4}(x,t)\right|}{|P_{\theta r/4}(x,t)|}\\
= \ &\frac{\left|\3w > u_0((1-\theta/2)^{-(n+\s)/\e}-1/2)\4 \cap P_{\theta r/4}(x,t)\right|}{|P_{\theta r/4}(x,t)|},\\
\leq \ &C(w(x,t) + (\theta r)^{-(n+\s)}(\theta r)^\s)^\e(u_0((1-\theta/2)^{-(n+\s)/\e} - 1/2))^{-\e}.
\end{align*}
By taking $\theta$ sufficiently small such that $(1-\theta/2)^{-(n+\s)} \geq 3/4$,
\begin{align*}
\frac{\left|\3u<u_0/2\4 \cap P_{\theta r/4}(x,t)\right|}{|P_{\theta r/4}(x,t)|}  &\leq C(w(x,t) + (\theta r)^{-n})^\e u_0^{-\e},\\
\leq & \ C[((1-\theta/2)^{-(n+\s)/\e}-1)^\e+(M^{-1}\theta^{-(n+\s)})^\e].
\end{align*}
At this point we just have to fix $\theta$ even smaller such that $C((1-\theta/2)^{-(n+\s)/\e}-1)^\e \leq 1/4$ which implies that for $M$ sufficiently large the right hand side above becomes smaller than $1/2$, independently of $\theta$.
\end{proof}

\section{Regularity results}\label{section:regularity}

H\"older estimates are obtained from the Harnack inequality. We note that the iteration at smaller scales should consider that the bound for the tail grows with an algebraic rate in each step. This is a technical difficulty that have been treated in previous papers, \cite{Caffarelli09}, \cite{Caffarelli11} or \cite{Davila12-p}. The bound in terms of the integral norm of $u$ is a consequence of the Oscillation Lemma.

\begin{corollary}\label{cor:holder}
Let $u$ satisfies
\begin{alignat*}{2}
 u_t - \cM^+_{\cL_0}u &\leq f(t) &&\text{ in $C_{1,1}$},\\
 u_t - \cM^-_{\cL_0}u &\geq -f(t) &&\text{ in $C_{1,1}$},
\end{alignat*}
Then there is some $\a \in (0,1)$ and $C>0$, depending only on $n$, $\l$, $\L$ and $\b$, such that for every $(y,s), (x,t) \in C_{1/2,1/2}$
\begin{align*}
\frac{|u(x,t) - u(y,s)|}{(|x-y| + |t-s|^{1/\s})^\a} \leq C\1\|u\|_{L^1((-1,0]\mapsto L^1(\w_\s))} + \|f\|_{L^1(-1,0)}\2.
\end{align*}
\end{corollary}

Further regularity in space can be obtained by imposing a smoothness condition on the kernels in order to average rough oscillations of the boundary data by an using an integration by parts technique as in \cite{Caffarelli09}. In this sense we let $\cL_1^\s(\l,\L,\b) := \{L_{K,b}^\s\} \ss \cL_0^\s(\l,\L,\b)$ such that,
\begin{align*}
|DK(y)| \leq \L|y|^{-1}.
\end{align*}

\begin{corollary}\label{cor:holder2}
Let $\cL \ss \cL_1$, $I:\R^\cL\to\R$ be uniformly elliptic, translation invariant and such that $I0=0$. Let $u$ satisfies
\begin{alignat*}{2}
u_t - Iu = f(t) \text{ in $C_{1,1}$}.
\end{alignat*}
Then there is some $\a \in (0,1)$ and $C>0$, depending only on $n$, $\l$, $\L$ and $\b$, such that for every $(y,s), (x,t) \in C_{1/2,1/2}$,
\begin{align*}
|Du(x,t)|+\frac{|Du(x,t)-Du(y,s)|}{(|x-y| + |t-s|^{1/\s})^\a} \leq C\1\|u\|_{L^1((-1,0]\mapsto L^1(\w_\s))} + \|f\|_{L^1(-1,0)}\2.
\end{align*}
\end{corollary}

The technique developed in \cite{Serra14} can be adapted to prove a $C^{1,\a}$ estimate in space for non-symmetric, rough kernels. It does not necessarily contain the previous result, because for $\s<1$ the H\"older exponent in the following Corollary is smaller than one.

\begin{corollary}
Let $\cL \ss \cL_0$, $I:\R^\cL\to\R$ be uniformly elliptic, translation invariant and such that $I0=0$. Let $u$ satisfies
\begin{alignat*}{2}
u_t - Iu = f(t) \text{ in $C_{1,1}$}.
\end{alignat*}
Then there is some $\a \in (0,1)$ and $C>0$, depending only on $n$, $\l$, $\L$ and $\b$, such that for every $(y,s), (x,t) \in C_{1/2,1/2}$,
\begin{align*}
|Du(x,t)|+\frac{|Du(x,t)-Du(y,s)|}{(|x-y| + |t-s|^{1/\s})^\a} \leq C\1\|u\|_{L^1((-1,0]\mapsto L^1(\w_\s))} + \|f\|_{L^1(-1,0)}\2.
\end{align*}
\end{corollary}

The proof follows the same ideas from \cite{Serra14} by compactness towards a global solution for which a Lioville type of Theorem has been already proven. We present some of the ideas of the original proof. Compared to the main Theorem in \cite{Serra14}, the right-hand side bounds are in terms of integral norms as a consequence of Corollary \ref{cor:holder2} and \ref{cor:oscillation1}.

\begin{proof}[Sketch of the proof]
It proceeds by contradiction by considering a sequence of translation invariant operators $\{I_k\} \ss \cL_0$ and normalized functions $\{u_k\}$ and $\{f_k\}$ such that,
\begin{align*}
&(u_k)_t - I_k u_k = f_k \text{ in $C_{1,1}$},\\
&\|u_k\|_{L^1((-1,0]\mapsto L^1(\w_\s))} + \|f_k\|_{L^1(-1,0)} \leq 1.
\end{align*}
By Corollary \ref{cor:oscillation1}, we get that $\|u_k\|_{L^\8(C_{3/4,3/4})} \leq C$. To state the contradiction hypothesis we consider the least squares fitting planes,
\begin{align*}
l_{k,r,x,t}(y) &:= a_{k,r,x,t}\cdot y + b_{k,r,x,t},\\
a_{k,r,x,t} &:= \frac{n+2}{r^2}\fint_{C_{r,r^\s}(x,t)}u(y,s)(y-x)dyds,\\
b_{k,r,x,t} &:= \fint_{C_{r,r^\s}(x,t)}u(y,s)dyds.
\end{align*}
By the proof of Lemma 4.3 in \cite{Serra14} (notice we just assume $\|u_k\|_{L^\8(C_{3/4,3/4})} \leq C$), the result follows if we get a contradiction by assuming,
\begin{align*}
\Theta(r) := \sup_{\substack{k\in\N\\\r\in(r,1/4)\\(x,t) \in C_{1/2,1/2}}} \frac{\|u_k-l_{k,\r,x,t}\|_{L^\8(C_{\r,\r^\s}(x,t))}}{\r^{1+\a}}  \underset{r\to 0}{\nearrow} \8.
\end{align*}
Given $r_i = 1/i$, for $i=5,6,\ldots$, there exist sequences $k_i \in \N$, $\rho_i \in (r_i,1/4)$, $(x_i,t_i) \in C_{1/2,1/2}$ such that,
\begin{align*}
\frac{\|u_{k_i-}l_{k_i,\r_i,x_i,t_i}\|_{L^\8(C_{\r_i,\r_i^\s}(x_i,t_i))}}{\r_i^{1+\a}} \geq \frac{\Theta(r_i)}{2} \geq \frac{\Theta(\r_i)}{2}.
\end{align*}
The intermediate inequality implies that $\r_i \to 0$ as $i\to\8$. Then we consider the rescalings,
\begin{align*}
v_i(x,t) := \1\frac{u_i- l_{k_i,\r_i,x_i,t_i}}{\r_i^{1+\a}\Theta(\r_i)}\2(x_i+\r_i x,t_i + \r^\s_i t).
\end{align*}
In particular they satisfy,
\begin{enumerate}
\item $\|v_i\|_{L^\8(\R^n\times(-1/2,0])} \geq 1/2$,
\item $\|v_i\|_{B_R\times(-1/2,0]} \leq CR^{1+\a}$ for $R\in(1,1/(2\r_i))$,
\item
\begin{align*}
\iint_{C_{1,1}} v_i(x,t)l(x)dxdt = 0, \text{ for all $l$ affine}.
\end{align*}
\item
\begin{align*}
(v_i)_t - \tilde I_i v_i = \tilde f_i \text{ in $C_{1/(2\r_i),1/2}$}
\end{align*}
where $\tilde I_i$ is uniformly elliptic with the same constants and $\|\tilde f_i\|_{L^1(-1,0)} \to 0$ as $i\to\8$.
\end{enumerate}

By the compactness resulting from the regularity of Corollary \ref{cor:holder} and the weak compactness of the operators, we recover the existence of accumulation points $v_\8$ and $I_\8$ such that,
\begin{enumerate}
\item $\|v_\8\|_{L^\8(\R^n\times(-1/2,0])} \geq 1/2$,
\item $\|v_\8\|_{B_R\times(-1/2,0]} \leq CR^{1+\a}$ for $R>1$,
\item
\begin{align*}
\iint_{C_{1,1}} v_\8(x,t)l(x)dxdt = 0, \text{ for all $l$ affine}.
\end{align*}
\item
\begin{align*}
(v_\8)_t - I_\8 v_\8 = 0 \text{ in $\R^n\times(-1/2,0]$}.
\end{align*}
where $I_\8$ is translation invariant and uniformly elliptic with the same constants.
\end{enumerate}
It suffices at this point to prove that such global solutions can only be an affine function depending only on the $x$ variable in order to get the contradiction from (2) and (3). Given the hypothesis we have at this points it follows exactly as in the proof of the Liouville Theorem 3.1 in \cite{Serra14} applying instead Corollary \ref{cor:holder} to the difference quotients.
\end{proof}

$C^{1,\a}$ regularity in time is not expected even if $I$ is translation invariant in time and $f\equiv 0$. A counterexample for the fractional heat equation was provided in \cite{Davila12-p}. The problem arises from the fact that sudden changes in time of the boundary data get immediately sensed by the equation given its non-local nature. Different to the previous estimates for the spatial derivatives, smooth kernels does not seem to control rough oscillations of the boundary data in this case. However, further regularity in time can be retrieved if the solution is controlled by,
\begin{align*}
[u]_{C^{0,1}((t_1,t_2]\mapsto L^1(\w_\s))} &:= \sup_{(t-\t,t] \ss (t_1,t_2]} \frac{\|u(t)-u(t-\t)\|_{L^1(\w_\s)}}{\t}.
\end{align*}

\begin{corollary}[Further regularity in time]\label{furthertime}
Let $\cL \ss \cL_0$, $I:\R^\cL\to\R$ be uniformly elliptic, translation invariant such that $I0=0$. Let $u$ satisfies
\begin{align*}
u_t-Iu = 0 \text{ in $C_{1,1}$}.
\end{align*}
Then there is some $\a \in (0,1)$ and $C>0$, depending only on $n$, $\l$, $\L$ and $\b$, such that for every $(x,t), (y,s)\in C_{1/2,1/2}$ we have 
\begin{align*}
|u_t(x,t)| + \frac{|u_t(x,t)-u_t(y,s)|}{(|x-y|+|t-s|^{1/\s})^\a}\leq C[u]_{C^{0,1}((-1,0]\mapsto L^1(\w_\s))}.
\end{align*}
\end{corollary}

The idea of the proof is to apply the Oscillation Lemma \ref{lem:oscillation} and the H\"older estimate \ref{cor:holder} to the difference quotients
\begin{align*}
w^\t(t) = \frac{u(t) - u(t-\t)}{\t} %\text{ for $\gamma = \a/\s, 2\a/\s,\ldots$}.
\end{align*}
By translation invariance, $w^\t$ satisfies in $C_{1,1}$,
\begin{align*}
w^\t_t - \cM_{\cL_0}^+ w^\t &\geq 0,\\
w^\t_t - \cM_{\cL_0}^- w^\t &\leq 0.
\end{align*}
By the Oscillation Lemma, $w^\t$ gets to be controlled by $[u]_{C^{0,1}((-1,0]\mapsto L^1(\w_\s))}$ in $C_{3/4,3/4}$, which implies the bound for $u_t$. Now we can proceed to truncate $w^\t = w_1^\t+w_2^\t$ where $w_1^\t = w^\t$ in $C_{5/8,5/8}$ and $w_1^\t = 0$ in $C_{3/4,3/4}$. The equations for $w^\t$ imply then equations for $w_1^\t$ with right hand sides controlled by $[u]_{C^{0,1}((-1,0]\mapsto L^1(\w_\s))}$. Corollary \ref{cor:holder} then completes the desired estimate.

%{\bf Acknowledgment:}

\bibliographystyle{plain}
\bibliography{mybibliography}

\end{document}